\documentclass[12pt]{amsart}
\usepackage{amsthm,amsmath,a4,amssymb,verbatim,xypic,enumerate}

\newtheorem{thm}{Theorem}[section]
\newtheorem{cor}[thm]{Corollary}
\newtheorem{lem}[thm]{Lemma}
\newtheorem{prop}[thm]{Proposition}
\newtheorem{fact}[thm]{Fact}
\theoremstyle{definition}
\newtheorem{defn}[thm]{Definition}
\theoremstyle{remark}
\newtheorem{rem}[thm]{Remark}
\numberwithin{equation}{section}
\newcommand{\PP}{\mathcal{P}}
\newcommand{\QQ}{\mathcal{Q}}
\newcommand{\MM}{\mathfrak{M}}
\newcommand{\Hom}{\textnormal{Hom}}
\newcommand{\Ass}{\textnormal{Asso}}

\newcommand{\Sub}{\textnormal{Sub}}

\newcommand{\eps}{\varepsilon}
\newcommand{\LL}{$(\mathsf{L}_2)$}
\newcommand{\LLL}{$(\mathsf{L}_{1+1})$}
\newcommand{\KK}{$(\mathsf{K}_2)$}
\newcommand{\KKK}{$(\mathsf{K}_{1+1})$}

\begin{document}

\address{CNRS and Univ Lyon, Univ Claude Bernard Lyon 1, Institut Camille Jordan, 43 blvd. du 11 novembre 1918, F-69622 Villeurbanne}
\email{cornulier@math.univ-lyon1.fr}
\subjclass[2010]{Primary 13C05; Secondary 06E15, 13E05, 13E10}

\title[Counting submodules]{Counting submodules of a module over a noetherian commutative ring}
\author{Yves Cornulier}%
\date{May 31, 2019}
\thanks{Supported by ANR JCJC CoCoGro. Partly undertaken during NPC programme in Newton Institute in Cambridge under EPSRC Grant EP/K032208/1}


\begin{abstract}
We count the number of submodules of an arbitrary module over a countable noetherian commutative ring. We give, along the way, various characterizations of minimax modules, as well as a structural description of meager modules, which are defined as those that do not have the square of a simple module as subquotient. We deduce in particular a characterization of uniserial modules over commutative noetherian rings. 
\end{abstract}
\maketitle
\section{Introduction}
All the rings in this paper are understood to be associative, unital and {\bf commutative} (unless explicitly stated), and countable sets are not assumed to be infinite.

An old classical result of Boyer \cite{Boy} (rediscovered 30 years later in \cite{RF}) characterizes those abelian groups whose set of subgroups is countable. Namely, these are abelian groups that are minimax and do not admit $C_{p^\infty}^2$ as a subquotient for any prime $p$. Here, $C_{p^\infty}$ denotes the Pr\"ufer group $\mathbf{Z}[1/p]/\mathbf{Z}$, and minimax means that it lies in an extension of abelian groups with a noetherian (or ``max") subgroup and artinian (or ``min") quotient. Keeping in mind that abelian groups are the same as $\mathbf{Z}$-modules, it is natural to seek generalizations of this results. A first immediate generalization can be done replacing $\mathbf{Z}$ with a countable principal ideal domain (PID):

\begin{prop}\label{pidcase}
Let $A$ be a countable PID and $M$ an $A$-module. Then the set of submodules of $M$ is countable if and only if $M$ is a minimax $A$-module and for every prime $\pi\in A$, $M$ does not admit $(A[\pi^{-1}]/A)^2$ as a subquotient.
\end{prop}

Let $K$ be a countable field. A natural example of a countable PID is the ring of polynomials $K[t]$. A $K[t]$-module is the same as a vector space over $K$ endowed with an operator $T$, and a submodule is the same as a subspace that is stable under $T$. So Proposition \ref{pidcase} tells when the number of $T$-stable subspaces is countable.

\begin{rem}\label{unconcase}In this context, the restriction to countable $A$ is natural. Indeed, if $K$ is an arbitrary field, say algebraically closed of infinite cardinal $\alpha$, the set of submodules of $A=K[t]$, that is, the set of ideals of $A$, is the free semigroup on the set $(t-s)_{s\in K}$, and its cardinal is $\alpha$. If we consider an arbitrary infinite subset $W$ in $K^*$ and consider the localization $A_W=K[t][(t-s)^{-1}:s\in K\smallsetminus W]$, then the set of ideals is now the free semigroup on the set $(t-s)_{s\in W}$, and its cardinal coincides with that of $W$. In particular, we can prescribe arbitrarily the (infinite) cardinal of the set of submodules of $A_W$ as a module over itself, with the only obvious restriction that it is $\le\alpha$.
\end{rem}

We omit the proof of Proposition \ref{pidcase}, because it is a straightforward adaptation of the case $A=\mathbf{Z}$, and because it follows as a particular case of the considerably more involved Theorem \ref{main} below.

The first main result of this paper is to tackle the case of countable noetherian rings. We were initially motivated by its special case of rings that are finitely generated, or finitely generated over a countable field. This essentially corresponds to abelian groups endowed with a finite commuting family of operators.

Before stating Theorem \ref{main}, we need to introduce and discuss a few properties, for an $A$-module. Recall that an $A$-module is noetherian (resp.\ artinian) if every ascending (resp.\ descending) chain of submodules stabilizes.

\begin{defn}An $A$-module $M$ is minimax if it has an artinian quotient with noetherian kernel.\end{defn}

This is an important finiteness condition, appearing in various contexts, and its definition is valid over an arbitrary associative ring $A$. Indeed, in our setting ($A$ commutative and noetherian), it admits the following characterizations, which, in the easier case of $A=\mathbf{Z}$, were observed (except (\ref{iomegan}) below) in \cite[Lemma 4.6]{BCGS}.
 
\begin{thm}\label{minidir}
Let $A$ be a noetherian ring and let $M$ be an $A$-module. The following are equivalent:
\begin{enumerate}[(i)]
\item\label{imini} $M$ is minimax;
\item\label{idires} $M$ does not admit an infinite direct sum of simple modules as subquotient;
\item\label{idire} $M$ does not admit an infinite direct sum of nonzero modules as subquotient;
\item\label{i2z} the poset $\Sub_A(M)$ contains no subposet isomorphic to $(2^\mathbf{Z},\subseteq)$;
\item\label{ichainr} there is no chain of submodules of $M$ that is order-isomorphic to $(\mathbf{R},\le)$;
\item\label{ichain} there is no chain of submodules of $M$ that is order-isomorphic to $(\mathbf{Q},\le)$;
\item\label{iomegan} there exists an integer $n$ such that no chain of submodules of $M$ is order-isomorphic to the ordinal $\omega^n$.
\end{enumerate}
\end{thm}

That minimax modules satisfy the other conditions (except maybe (\ref{iomegan})) is essentially immediate and the converse, namely (\ref{idires})$\Rightarrow$(\ref{imini}) takes a little more work. See \S\ref{smx} for the proof of Theorem \ref{minidir}. Note that (\ref{iomegan}), unlike the other ordered conditions, is not symmetric under reversing the order: indeed if $A$ is a noetherian ring of infinite Krull dimension, then, despite being a minimax module over itself, it admits a chain of ideals that is reverse-isomorphic to the countable ordinal $\omega^\omega=\sup_n\omega^n$ \cite[Theorem 2.12]{Bass}. Furthermore, there exist \cite{GoRo,Gul2} noetherian rings $A$ of uncountable Krull dimension; by the same result of Bass, they admit, for every countable ordinal $\alpha$, a chain of ideals that is reverse-isomorphic to $\alpha$. Also note that if $A$ has finite Krull dimension $d$, it follows from the proof that $n$ can be replaced, independently of $M$, with $d+1$ in~(\ref{iomegan}). In the case $A$ is local and complete, there is one more characterization of $M$ being minimax, namely that $M$ is isomorphic to its Matlis bidual, see \S\ref{matlis}.

\begin{defn}
An $A$-module $M$ satisfies Property {\LL} if for every artinian quotient $Q$ of $M$, the poset of submodules of $Q$ does not contain any chain isomorphic to $\omega^2$.
\end{defn}

Despite its simplicity, this definition is highly non-explicit in terms of the structure of $M$: a more natural definition is given later, based on the notion of Loewy dimension. See the characterization in Corollary \ref{l2cara}. Let us just mention that the ring $A$ has Krull dimension $\le 1$ if and only if every $A$-module satisfies Property {\LL}.

If $M$ does not satisfy Property {\LL}, then it admits a subquotient $Q$ that is artinian and fails to satisfy Property {\LL}, while all proper submodules of $Q$ satisfy Property {\LL}. This subquotient cannot always be chosen to be a quotient, see Remark \ref{noquo2}. 

Property {$(\mathsf{L}_2)$} is stable under taking submodules, quotient modules and extensions (the case of submodules is not trivial; one can for instance use Corollary \ref{loewydec}). Note that the definition of Property {$(\mathsf{L}_2)$} only depends on the poset $\textnormal{Sub}_A(M)$.

Being both minimax and satisfying Property {$(\mathsf{L}_2)$} can be rephrased in a convenient way.

\begin{prop}\label{l2intro}
Let $A$ be a noetherian ring and let $M$ be an $A$-module. The following are equivalent.
\begin{enumerate}[(i)]
\item\label{l2_1} $M$ satisfies {\LL} and is minimax; 
\item\label{l2_2} $M$ admits a composition series $0\subset M_0\subset M_1\subset\dots\subset M_n=Q$ such that $M_0$ is finitely generated, and for every $i\ge 1$, every proper submodule of $M_i/M_{i-1}$ has finite length;
\item\label{l2_3} the poset of submodules of $M$ has no chain that is order-isomorphic to $\omega^2$.
\end{enumerate}
\end{prop}

The third property we have to emphasize is the following:

\begin{defn}
An $A$-module $M$ satisfies {\LLL} if no subquotient of $M$ is isomorphic to the square $P\oplus P$ of some artinian $A$-module $P$ of infinite length.
\end{defn}

A characterization among minimax modules is as follows:

\begin{prop}
Let $A$ be a noetherian ring and let $M$ be a minimax $A$-module. Then $M$ fails to satisfy {\LLL} if and only if for every finitely generated submodule $N$ such that $Q=M/N$ is artinian, $Q$ admits a submodule isomorphic to $P\oplus P$, where $P$ has infinite length but all its proper submodules have finite length.
\end{prop}

Among minimax modules, unlike Property {\LL}, Property {\LLL} cannot be characterized in terms of the order type of chains of the poset of submodules. On the one hand, over $A=\mathbf{Z}$, for distinct primes $p,q$, the order types of the chains of submodules of $C_{p^\infty}\oplus C_{q^\infty}$ and $C_{p^\infty}^2$ are the same (these are precisely all chains isomorphic to a subchain of $\omega+\omega$). On the other hand, $C_{p^\infty}\oplus C_{q^\infty}$ satisfies {\LLL} while $C_{p^\infty}^2$ fails to satisfy {\LLL}. Nevertheless, Property {\LLL} can be characterized in terms of the poset of submodules, see Corollary \ref{posetlll}.

We can now state our main result.

\begin{thm}\label{main}
Let $A$ be a countable noetherian (commutative) ring. Let $M$ be an $A$-module. Then the set of submodules of $M$ is countable if and only if $M$ is minimax and satisfies Properties {\LL} and {\LLL}.

Otherwise, $M$ has at least $2^{\aleph_0}$ submodules, with equality if $M$ is countable.
\end{thm}

Under the setting of Theorem \ref{main}, if $M$ is uncountable, it was established by Burns, Okoh, Smith and Wiegold \cite{BOSW} that the number of submodules of $M$ is the largest possible, namely $2^{\textnormal{card}(M)}$.  The case considered here, namely that of countable modules, is more delicate, as we see that the cardinality of $\Sub_A(M)$ is not governed by the cardinality of $M$ (see also Remark \ref{rembosw}). 

The proof of Theorem \ref{main}, concluded in \S\ref{promain} makes uses of Matlis duality (see \S\ref{remi} and \S\ref{matlis}), which is instrumental in the understanding of artinian modules over noetherian rings. Namely, it relates artinian $A$-modules with finitely generated modules over the completions of its localizations at maximal ideals, and the proof of Theorem \ref{main} is based on the following, proved in \S\ref{ficom}.

\begin{thm}\label{thmB}
Let $R$ be a noetherian complete local ring with countable residual field. Let $M$ be a finitely generated $R$-module. Then the set of submodules of $M$ is countable if and only if the following two conditions are satisfied
\begin{itemize}
\item[{\KK}] $M$ has Krull dimension at most one;
\item[{\KKK}] for every prime ideal $\PP$ of coheight one, $(A/\PP)^2$ is not isomorphic to any subquotient of $M$.
\end{itemize}
Otherwise, $M$ has $2^{\aleph_0}$ submodules.
\end{thm}

(To say that $\mathcal{P}$ has coheight one means that it is maximal among non-maximal prime ideals.)

These two conditions are independent; of course {\KKK} does not follow from {\KK}, and conversely we check (see Proposition \ref{deuxd}) that if $M=R$ is a UFD of Krull dimension 2 (thus failing to satisfy {\KK}), then $M$ satisfies {\KKK}. An example is $M=R=\mathbf{k}[\![X,Y]\!]$ when $\mathbf{k}$ is a field.

\begin{rem}
Properties {\LLL}, {\LL} and being minimax are three independent conditions as well:
\begin{itemize}
\item the above example $C_{p^\infty}^2$, is a minimax $\mathbf{Z}$-module (as it is artinian), satisfies {\LL} but not {\LLL};
\item a direct sum of infinitely many simple $A$-modules satisfies {\LL}, {\LLL}, but is not minimax;
\item the injective hull of the residual field of the localization $A=\mathbf{k}[X,Y]_0$ is a minimax $A$-module (as it is artinian), satisfies {\LLL} but not {\LL}. This follows from the previous example ($M=R=\mathbf{k}[\![X,Y]\!]$) by Matlis duality.
\end{itemize}
\end{rem}

When $A$ has Krull dimension $\le 1$, Property {\LL} becomes an empty condition, and thus becomes superfluous in Theorem \ref{main}, which then simplifies to an easier statement closer to that of Proposition \ref{pidcase}. 

A refinement of the problem consists in describing the topological type (and not only the cardinality) of the space $\Sub_A(M)$ of submodules of an $A$-module $M$, viewed as a closed (and thus compact) subset of the power set $2^M$. This task was carried out
\begin{itemize}
\item In \cite{Co}, restricting to finitely generated modules;
\item in \cite{CGP2}, restricting to $A=\mathbf{Z}$ but with no finite generation assumption.
\end{itemize}
It would be interesting to carry this task over for an arbitrary module over a finitely generated (commutative) ring. Theorem \ref{main} is a first step towards this direction, as well as Corollary \ref{exisola}, which says in particular that in this case, $\Sub_A(M)$ has an isolated point if and only of $M$ is minimax.

Recall that a topological space is scattered if every non-empty subset has an isolated point; for a compact Hausdorff space this is equivalent to being countable. Using refinements of the proof of Theorem \ref{main}, we can prove (Proposition \ref{nosca}) that if $\Sub_A(M)$ is scattered, then $M$ is minimax and satisfies Properties {\LL} and {\LLL}. We do not know whether the converse holds (if true, it would provide a generalization of Theorem \ref{main} without the assumption that $A$ is countable).

In view of Remark \ref{unconcase}, the question of the cardinality is, in my opinion, less interesting when $A$ is uncountable. In general, it should involve a discussion on the cardinality of simple $A$-modules. See however Theorem \ref{prop3} below. Before stating it, we need the following definition: we say that an $A$-module is {\em meager} if it does not admit $S^2$ as a subquotient for any simple $A$-module $S$.

Let us now provide a structural result for meager modules. Recall that two ideals $I,J$ of $A$ are disjoint if $I+J=A$; another terminology is ``comaximal". The following theorem shows that, in the commutative noetherian case, there are very strong restrictions on the possible structure of meager modules.

\begin{thm}\label{imeager_stru}
1) Let $A$ be a noetherian ring; let $M$ be a meager $A$-module. Then $M$ has a unique decomposition $M=\bigoplus_{\PP\in\Ass_A(M)} M(\PP)$, where $\Ass_A(M(\PP))=\{\PP\}$, and the $\PP\in\Ass_A(M)$ are pairwise disjoint. Conversely, given a subset $X$ of pairwise disjoint prime ideals of $A$ and for each $\PP\in X$ a meager $A$-module $M(\PP)$ with $\Ass_A(M(\PP))=\{\PP\}$, the direct sum $\bigoplus_{\PP\in X}M(\PP)$ is meager.

2) Let $M$ be a meager $A$-module with a single associated prime $\PP$. Then exactly one of the following holds:
\begin{enumerate}[(a)]
\item\label{imea1i} $M$ has nonzero finite length and its submodules form a chain;
\item\label{imea2i} $\PP$ is a maximal ideal, and there exists a (unique) prime ideal $\QQ$ of coheight 1 in the completed local ring $\widehat{A_\PP}$ such that $B=A_\PP/\QQ$ is a discrete valuation ring and $M$ is isomorphic to $\mathrm{Frac}(B)/B$ as an $A$-module;
\item\label{imea3i} $\PP$ has coheight 1, the quotient ring $A/\PP$ is a Dedekind domain, and $M$ is a torsion-free module of rank 1 over $A/\PP$ (or equivalently, is isomorphic to some nonzero submodule of $\mathrm{Frac}(A/\PP)$).
\end{enumerate}
Conversely, any $A$-module in one of these cases is meager with only associated ideal $\PP$.
\end{thm}

See \S\ref{s_meager}, notably Proposition \ref{assdis2} and Theorem \ref{meager_stru}, which encompass Theorem \ref{imeager_stru}.

In Case (\ref{imea1i}), $M$ is a cyclic module and can therefore been seen as an artinian local ring quotient $A/I$, which is a principal ideal ring. A result of Hungerford \cite{hun} then says that the ring $A/I$ is isomorphic to a quotient of some complete discrete valuation ring (which is not an $A$-algebra in general). In Case (\ref{imea2i}), beware $\QQ$ is not necessarily related to an ideal of coheight 1 of $A$; more precisely, the inverse image of $\QQ$ in $A$ can have coheight greater than 1.
 
We use Theorem \ref{imeager_stru} to prove the following counting result, which also goes beyond the countable case. In view of the pathological examples of Remark \ref{unconcase}, it involves a restriction on the ring that is satisfied in many cases, notably all noetherian algebras that are countably generated over an infinite field, and their localizations. 

\begin{thm}\label{prop3}
Let $A$ be a noetherian ring of cardinal $\alpha$. Assume that every quotient field of $A$ also has cardinal $\alpha$. Let $M$ be a minimax $A$-module.
Then
\begin{itemize}
	\item if $M$ is not meager, the number of its submodules is
		\begin{itemize}
   			\item $\alpha$ if $M$ satisfies {\LL} and {\LLL};
			\item $\alpha^{\aleph_0}$ if {\LL} or {\LLL} fails.
		\end{itemize}
	\item suppose that $M$ is meager:
		\begin{itemize}
			\item if $M$ has finite length, it has finitely many submodules;
			\item if $M$ has infinite length and all its associated ideals are maximal, then the number of its submodules is $\aleph_0$;
			\item otherwise, let $\beta$ be the maximum over all associated non-maximal ideals $\PP$ of $M$ of the number of maximal ideals containing $\PP$ (note that $\beta\le\alpha$); then the number of submodules of $M$ is $\max(\beta,\aleph_0)$;
		\end{itemize}
\end{itemize}
\end{thm}

Another corollary of Theorem \ref{imeager_stru} is the following characterization of uniserial modules over arbitrary noetherian rings. Recall that a module over a ring is {\em uniserial} if its submodules form a chain (i.e., is totally ordered under inclusion). It is straightforward that every uniserial module is meager with at most one associated prime ideal. We deduce:

\begin{cor}\label{uniserial}
Let $A$ be a noetherian ring and let $M$ be a nonzero $A$-module. Then $M$ is uniserial if and only if 
one of the following holds:
\begin{enumerate}[(a)]
\item $M$ has finite length; all its submodules are cyclic, and it has a single associated prime ideal $\PP$, which is maximal;
\item\label{icha2} $M$ is artinian of infinite length: for some maximal ideal $\PP$ and some prime ideal $\QQ$ of the completion $\widehat{A_\PP}$ such that $B=A_\QQ/\PP A_\QQ$ is a discrete valuation ring, the $A$-module $M$ is isomorphic to $\mathrm{Frac}(B)/B$;
\item\label{icha3} $M$ is not artinian: for some non-maximal prime ideal $\PP$ such that the quotient ring $B=A/\PP$ is a discrete valuation ring, the module $M$ is isomorphic to either $B$ or $\mathrm{Frac}(B)$ as $A$-module.
\end{enumerate}
\end{cor}

\begin{rem}\label{rembosw}
In \cite{BOSW}, the proof that $M$ has $2^{\mathrm{card}(M)}$ distinct submodules amounts to showing that an uncountable module $M$ always possesses a submodule isomorphic to a direct sum of $\mathrm{card}(M)$ nonzero modules. When $M$ is countable, the existence of a subquotient of $M$ isomorphic to an infinite (countable) direct sum of nonzero modules is equivalent to the failure of being minimax (Theorem \ref{minidir}). This explains why minimax modules are often the most subtle case when we study the set of submodules. 
\end{rem}

As regards counting submodules, let us mention results in other directions; for instance, in the non-noetherian setting, Steprans \cite{Ste} shows that the statement ``every uncountable module $M$ over a countable commutative ring has $2^{\mathrm{card}(M)}$ submodules" is undecidable in ZFC. Also, in the non-commutative case, he provides a countable ring with a module with exactly $\aleph_1$ (the minimum possible) many submodules.

Going back to Theorem \ref{main}, there are natural questions left. A first one is what happens in the case of a non-noetherian (countable) ring $A$. In this case, even the case $M=A$ is delicate (while the case of $M$ finitely generated is trivial in case $A$ is noetherian). Understanding which finitely generated modules have countably many submodules already seems delicate, if not out of reach, in the case when $A$ is the group ring of a free $\mathbf{Z}$-module of infinite rank. On the other hand, it would be interesting to study some cases when $A$ is noetherian but non-commutative. This includes the case of group rings of virtually polycyclic groups; their study would require specific methods also going beyond the scope of this paper.

\medskip
{\bf Outline.} \S\ref{remi} includes classical preliminaries, notably addressing Matlis duality and Loewy dimension of artinian modules. Elaborations on these, including proofs, are included in \S\ref{addpre}. Results concerning finitely generated modules over complete local rings are gathered in \S\ref{ficom}. The proof of Theorem \ref{main} is concluded in \S\ref{promain}. Meager modules are addressed in \S\ref{megmo}.

\medskip

{\bf Acknowledgment.} I thank Mark Sapir for his valuable advice to improve the presentation of the paper.


\section{Preliminaries}\label{remi}

\subsection*{Associated ideals}
Let $A$ be a noetherian ring and $M$ an $A$-module. Recall that $\Ass_A(M)$ is defined as the set of prime ideals $\PP$ of $A$ such that $A/\PP$ embeds as a submodule of $M$; these are called associated prime ideals of $M$. It is known to be non-empty if $M\neq 0$, and finite if $M$ is noetherian.

In particular, every associated prime ideal of $M$ is maximal if and only if $M$ is {\it locally of finite length}, in the sense that every finitely generated submodule of $M$ has finite length.

\subsection*{Artinian modules}

An important class of modules is the class of artinian modules, namely in which there is no strictly decreasing sequence of submodules. They are locally of finite length.

Given an $A$-module $M$, there is, for every maximal ideal $\mathcal{M}$ of $A$, a natural homomorphism $M\to M\otimes_AA_\mathcal{M}$, giving rise to a product homomorphism $\phi:M\to \prod_\mathcal{M}M\otimes_AA_\mathcal{M}$. If $M$ is artinian, this product involves finitely many nonzero terms (because $\Ass_A(M)$ is finite and $M\otimes_AA_\mathcal{M}=\{0\}$ if $\mathcal{M}\notin\Ass_A(M)$) and $\phi$ is an isomorphism, yielding a canonical finite decomposition
\[M=\bigoplus_{\mathcal{M}\in\Ass_A(M)}M(\mathcal{M}),\quad\Ass_A(M(\mathcal{M}))=\{\mathcal{M}\}.\]
Moreover, $M(\mathcal{M})\simeq M\otimes_AA_\mathcal{M}$ is naturally a module over the completion $\widehat{A_\mathcal{M}}$. (See Lemma \ref{assdis} for an extension of this decomposition for arbitrary modules locally of finite length.)

\subsection{Krull dimension $\ge 2$}

We write for reference the following well-known consequence of the Hauptidealsatz.
\begin{lem}\label{k2}
Let $(R,\MM)$ be a noetherian local ring of Krull dimension at least two. Then for any $x\in\MM$ and $n\ge 1$, the ideal $Rx+\MM^{n+1}$ does not contain $\MM^n$. (Equivalently, the dimension of $\MM^n/\MM^{n+1}$ as vector space over $R/\MM$ is $\ge 2$.)
\end{lem}
\begin{proof}
If $R_1=R/xR$ and $\MM_1$ is the image of $\MM$, Krull's Hauptidealsatz implies that $R_1$ has Krull dimension at least one. But the hypotheses imply that $\MM_1^n=\MM_1^{n+1}$, which implies that $R_1$ is artinian, a contradiction.
\end{proof}

\subsection*{$\MM$-adic topology of complete local rings}

Let $(R,\MM)$ be a noetherian local ring. The $\MM$-adic topology on a finitely generated $R$-module $M$ (and in particular on $R$) is the topology induced by the ultrametric distance \[d(x,y)=\exp(-\sup\{n:x-y\in\MM^nM\}).\] This is indeed a distance (and not only a semi-distance), since Krull's intersection theorem says that $\bigcap\MM^nM=0$. 

\begin{prop}\label{isclosed}
Every submodule $N$ of a finitely generated $R$-module $M$ is closed in the $\MM$-adic topology. In particular, every ideal of $R$ is closed.
\end{prop}
\begin{proof}
Indeed, we have $N=\bigcap_n (N+\MM^nM)$, by Krull's intersection theorem (applied to the module $M/N$).
\end{proof}

\subsection*{Matlis duality}
Let $R$ be a complete noetherian local ring. Let $E$ be an injective hull of the residual field $\mathbf{k}$ of $R$. Then $T(-)=\Hom(-,E)$ is a contravariant functor, and $T(M)$ is called the {\em Matlis dual} of $M$. 
Matlis duality can be stated as follows. 
\begin{thm}\label{matdu}
1)If $M$ is noetherian (=finitely generated) then $T(M)$ is artinian; if $M$ is artinian then $T(M)$ is noetherian. In both cases, The canonical homomorphism $M\to T(T(M))$ is bijective.

2) Matlis duality $M\mapsto T(M)$ establishes a contravariant equivalence between the categories of noetherian and artinian $R$-modules. It restricts to a contravariant self-equivalence of the category of $R$-modules of finite length.
\end{thm}

See \cite[\S 3.2]{BH}. We deduce a more general statement in the context of minimax modules in \S\ref{matlis}.

Matlis duality is useful to describe artinian modules $M$ over a given noetherian ring $A$. Indeed, such $M$ decomposes as a finite direct sum $M=\bigoplus_{\MM\in\Ass_A(M)}M(\MM)$, and we can view $M(\MM)$ as an artinian module over the completion $\widehat{A_\MM}$. To summarize, the artinian $A$-modules correspond under Matlis duality to finitely supported families of noetherian modules over the various completions of localizations at maximal ideals of $A$.

\subsection*{Loewy dimension}
If $M=\bigoplus M_\mathcal{M}$ is an artinian $A$-module as above, its {\em Loewy dimension} is defined as the supremum over $\mathcal{M}$ of the Krull dimension of the Matlis dual of $M_\mathcal{M}$ (viewed as $\widehat{A_\mathcal{M}}$-module). It is a finite number, because the Krull dimension of any local ring is finite, see \cite[Theorem~13.3]{matsumura}.

The definition of Loewy dimension extends to any $A$-module $M$; namely its Loewy dimension is the supremum of the Loewy dimensions of all artinian quotients of $M$ (if infinite this is $\omega$). In particular, for a minimax $A$-module, it equals the Loewy dimension of $M/N$ where $N$ is any finitely generated submodule such that $M/N$ is artinian. (By convention the Loewy dimension of $\{0\}$ is $-\infty$, so if $M$ is finitely generated we need to impose $N\neq M$.) The Loewy dimension is obviously monotone with respect to taking quotients, but also with respect to taking submodules (Corollary \ref{loewydec}). Note that the Loewy dimension of an $A$-module is bounded above by the Krull dimension of $A$. 

Alternatively, it is possible to avoid the use of Matlis duality to define it. Indeed, recall that, as observed by Bass \cite{Bass}, given a noetherian local ring $R$ and a finitely generated $R$-module $M$ of Krull dimension $d$ (so $d$ is finite), then $d$ is the largest $k$ such that the ordinal $\omega^k$ has a decreasing embedding into the chain of submodules of $M$. Therefore, using Corollary \ref{cmatlis}, given a noetherian ring $A$ and a minimax $A$-module $M$, its Loewy dimension is the largest $k$ such that $\omega^k$ has an increasing embedding into the chain of submodules of $M$. 
This allows to characterize Property {\LL} of the introduction in terms of Loewy dimension, see Corollary \ref{l2cara}.

\subsection*{Ordinal length}
Let $A$ be a ring (not necessarily commutative) and $M$ a noetherian $A$-module. The ordinal length $\ell(M)$ of $M$ is defined inductively as follows
 $$\ell(M)=\sup\{\ell(N)+1\},$$ where $N$ ranges over proper quotients of $M$.
The reader can check that this definition (due to Gulliksen, see \cite{Gull}) is consistent, and extends the usual notion of length for modules of finite length; see also \S\ref{s_lomega}.  

\subsection*{Perfect sets}

Recall that a topological space is perfect if it has no isolated point. 

\begin{lem}\label{ontocantor}
Every nonempty compact Hausdorff space without isolated points, having a basis of clopen subsets, has a continuous map onto a Cantor set, and in particular contains $2^{\aleph_0}$ points.\qed
\end{lem}
This is well-known and the easy proof is left to the reader. 
(For a compact Hausdorff space, to have a basis of clopen subsets is equivalent to being totally disconnected. Nevertheless, we do not need this equivalence as we only use it for closed subsets of $2^X$, where $X$ is a discrete set, which have an obvious basis of clopen subsets.) It is also known that the cardinality fact holds for arbitrary perfect compact Hausdorff spaces, but this is a little harder and we do not need this here.

\section{Additional preliminaries}\label{addpre}
Throughout this section, $A$ denotes a noetherian (commutative) ring.

\subsection{Characterizations of artinian modules}

Recall that a module over a ring (associative unital, not necessarily commutative) is {\em finitely cogenerated} if it satisfies one of the following equivalent conditions (see \cite[Prop.\ 19.1]{lam} and \cite{vam}):

\begin{enumerate}
\item the socle (= submodule generated by simple submodules) of $M$ has finite length and is an essential submodule (i.e., has nonzero intersection with every nonzero submodule);
\item $M$ admits an essential submodule of finite length;
\item the intersection of every nonempty chain of nonzero submodules is nonzero;
\item $M$ is isomorphic to a submodule of an injective hull of a module of finite length.
\end{enumerate}

In this generality, it is easy to check that every artinian module is finitely cogenerated, and actually that a module is artinian if and only if all its quotients are finitely cogenerated. In the commutative noetherian case, the following result of V\'amos \cite{vam}, based on Matlis duality, holds:

\begin{thm}\label{arti}
An $A$-module $M$ is artinian if and only if it is finitely cogenerated. This holds if and only if it satisfies the following three conditions.
\begin{enumerate}[(a)]
\item\label{ar1} $M$ is locally of finite length (i.e., every associated prime ideal of $M$ is maximal);
\item\label{ar3} $M$ has finitely many associated prime ideals;
\item\label{ar2} for every maximal ideal $\MM$ of $A$, $\mathrm{Hom}(A/\MM,M)$ has finite dimension over $A/\MM$.\qed
\end{enumerate}
\end{thm}

Note that the conjunction of the conditions (\ref{ar1})-(\ref{ar3})-(\ref{ar2}) is clearly equivalent, in the commutative case, to the condition that the socle has finite length and is essential, which is one of the above characterizations of being finitely cogenerated.

\begin{cor}\label{resart}
Every $A$-module $M$ is residually artinian.
\end{cor}
\begin{proof}
Let $x$ be a nonzero element of $M$. Let $N$ be a maximal submodule of $N$ for the condition $x\notin N$. We have to show that $M/N$ is artinian. Indeed, in $M/N$, the cyclic submodule $(Ax+N)/N$ contains every nonzero submodule, so is finitely cogenerated, and hence Theorem \ref{arti}. (Note that in greater generality -- no commutativity -- this shows that every module is residually finitely cogenerated.)
\end{proof}

\begin{cor}\label{loewydec}
For every $A$-module $M$ and submodule $N$, every artinian quotient of $N$ embeds into an artinian quotient of $M$. In particular, the Loewy dimension of $N$ is bounded above by the Loewy dimension of $M$.
\end{cor}
\begin{proof}
We can suppose that $N$ is artinian. Let $P$ be a maximal submodule among those having zero intersection with $N$. Then $N$ is an essential submodule of $M/P$, and hence so is the socle of $N$. Thus $M/P$ is artinian.
\end{proof}

Given an $A$-module $M$, a subset $X$ of $M$ is called a {\em discriminating subset} if every nonzero submodule of $M$ contains a nonzero element of $X$. The module $M$ is called {\em finitely discriminable} if it has a finite discriminating subset. Equivalently, this means that it contains finitely many nonzero submodules $M_1,\dots,M_k$ such that every nonzero submodule of $M$ contains one of the $M_i$. The following proposition was checked by Yahya \cite{yahya} in case $A$ is the ring $\mathbf{Z}$ of integers (see also \cite[Lemma~4.1]{CGP}).

\begin{prop}\label{fd}
Every finitely discriminable $A$-module is artinian, and the converse holds if and only if every simple $A$-module is finite.
\end{prop}
\begin{proof}
Suppose $M$ finitely discriminable. It easily follows that every associated prime ideal of $M$ is maximal, that $\Ass_A(M)$ is finite, and that for every maximal ideal $\MM$, the $\MM$-torsion of $M$ is finite. By Theorem \ref{arti}, this forces $M$ to be artinian.

If the module $M$ is locally of finite length, then its socle, namely the submodule $N$ generated by simple submodules, is discriminating. If $M$ is artinian then $N$ has finite length, and if we assume in addition that every simple $A$-module is finite, we deduce that $N$ is a finite discriminating subset.

Finally if $\mathbf{k}$ is an infinite simple $A$-module (and thus can be thought of as a quotient field of $A$), then $\mathbf{k}^2$ is artinian but is obviously not finitely discriminable as for every finite subset $F$ of $\mathbf{k}^2-\{0\}$ we can find a 1-dimensional $\mathbf{k}$-subspace of $\mathbf{k}^2$ disjoint from $F$.
\end{proof}

Endow the set $\Sub(M)$ of submodules of an $A$-module $M$ with the topology induced by inclusion in the compact set $2^M$. Then we also have the following corollary.

\begin{cor}\label{exisola}
Let $M$ be an $A$-module.
\begin{enumerate}
\item If $N$ is an isolated point in $\Sub(M)$, then $N$ is finitely generated and $M/N$ is artinian. In particular, if $\Sub(M)$ has an isolated point then $M$ is minimax.
\item Assuming that every simple $A$-module is finite, the converse holds: a submodule $N$ is an isolated point in $\Sub(M)$ if and only if $N$ is finitely generated and $M/N$ is artinian, and $\Sub(M)$ has a isolated point if and only $M$ is minimax.
\end{enumerate}
\end{cor}
\begin{proof}
It is straightforward that $N$ is isolated in $\mathrm{Sub}(M)$ if and only if $N$ is finitely generated and $M/N$ is finitely discriminable. So the result follows from Proposition \ref{fd}.
\end{proof}

\subsection{Minimax modules}\label{smx}

We start with the easy implications in Theorem \ref{minidir}.

\begin{fact}The implications (\ref{imini})$\Rightarrow$(\ref{ichain})$\Rightarrow$(\ref{ichainr})$\Rightarrow$(\ref{i2z})$\Rightarrow$(\ref{idire})$\Rightarrow$(\ref{idires}) and (\ref{iomegan})$\Rightarrow$(\ref{ichain}) hold (for modules over an arbitrary associative unital ring).
\end{fact}
\begin{proof}
(\ref{idire})$\Rightarrow$(\ref{idires}) is trivial.

(\ref{i2z})$\Rightarrow$(\ref{idire}) is done by contraposition: passing to a subquotient we can suppose that $M=\bigoplus_{n\in\mathbf{Z}}M_n$ with $M_n\neq 0$ and we map $I\subset\mathbf{Z}$ to the partial sum $\bigoplus_{n\in I}M_n$.

(\ref{ichainr})$\Rightarrow$(\ref{i2z}), done by contraposition, follows from the observation that the poset $(\mathbf{R},\le)$ embeds as a subposet of $(2^{\mathbf{Q}},\subseteq)$, mapping $r$ to $[r,+\infty\mathclose[\cap\mathbf{Q}$.

(\ref{ichain})$\Rightarrow$(\ref{ichainr}) is trivial (and actually its converse (\ref{ichainr})$\Rightarrow$(\ref{ichain}) holds in a wide generality, because the poset of submodules is complete).

(\ref{iomegan})$\Rightarrow$(\ref{ichain}) is trivial since the ordered set $(\mathbf{Q},\le)$ contains copies of $\omega^n$ for all $n$.

(\ref{imini})$\Rightarrow$(\ref{ichain}) Let by contradiction $(M_q)_{q\in\mathbf{Q}}$ be such a chain. Let $N$ be a noetherian submodule of $M$ such that $M/N$ is artinian. Since $N$ is noetherian, there exists a rational $q$ such that for all $r\ge q$ we have $M_q\cap N=M_r\cap N$. Since $M/N$ is artinian, there exists a rational $r>q$ such that for all rational $s$ with $q<s<r$ we have $M_s+N=M_r+N$. It follows that $M_s=M_r$ for all such $s$, contradicting the injectivity of $r\mapsto M_r$.
\end{proof}

\begin{prop}
Implication (\ref{imini})$\Rightarrow$(\ref{iomegan}) of Theorem \ref{minidir} holds.
\end{prop}
\begin{proof}
(\ref{imini})$\Rightarrow$(\ref{iomegan}) Suppose that $M$ is minimax. Let $N$ be a noetherian submodule such that $M/N$ is artinian and define $n$ so that $n-1$ is the Loewy dimension of $M/N$. Consider a strictly increasing chain $(M_i)_{i\in\omega^n}$ of submodules. Since $N$ is noetherian, $(M_i\cap N)$ is stationary and since any nonempty right segment of $\omega^n$ contains a copy of $\omega^n$, we can assume that $M_i\cap N$ is equal, for all $i$, to a single submodule $P$ of $N$. Hence, since for $i<j$ $M_i\subsetneq M_j$, we also have $M'_i\subsetneq M'_j$, where $M'_i$ is the projection of $M_i$ in $M/N$. This reduces to the case when $M$ is artinian, as we suppose now.

The artinian module $M$ decomposes canonically as a direct product of artinian modules with a single associated prime ideal (which is maximal), and hence the same argument shows that at least one of this summands, say $N$, with associated ideal $\MM$, contains a chain of submodules isomorphic to $\omega^n$. Viewing $N$ as $R$-module, for $R=\widehat{A_\MM}$, we can apply Matlis duality and hence the Matlis dual, as a finitely generated $R$-module, has Krull dimension $\le n-1$. Then it is a classical result of Bass \cite[Theorem 2.12]{Bass} that it cannot have a chain of submodules reverse-isomorphic to $\omega^n$.
\end{proof}

Recall that an $A$-module is semisimple if it is isomorphic to a (possibly infinite) direct sum of simple modules. Thus the negation of (\ref{idires}) in Theorem \ref{minidir} precisely means the existence of a semisimple quotient of infinite length. It is easily seen that a module $M$ is semisimple if and only if it is locally of finite length, and for every maximal ideal $\MM$ and $x\in M$, $\MM^2x=0$ implies $\MM x=0$. It follows that an increasing union of semisimple modules is semisimple (this latter fact actually holds in a non-commutative context, see \cite[Theorem 2.4]{lam}).

We can now conclude the proof of Theorem \ref{minidir}.

\begin{prop}\label{exiqtsemi}
The implication (\ref{idires})$\Rightarrow$(\ref{imini}) in Theorem \ref{minidir} holds: for every (commutative) noetherian ring and non-minimax $A$-module $M$, there exists an infinitely generated semisimple subquotient of $M$.
\end{prop}
\begin{proof}
Let $M$ be a non-minimax $A$-module. We will distinguish two cases. First assume that $M$ has a finitely generated submodule $N$ such that $M/N$ is locally of finite length. Then since $M$ is not minimax, $M/N$ is also not minimax. So either (\ref{ar3}) or (\ref{ar2}) of Theorem \ref{arti} fails for $M/N$, and this implies that $M/N$ admits an infinite direct sum of nonzero modules as submodule. 

Now let us treat the last case, namely when $M$ has no finitely generated submodule such that the quotient is locally of finite length.

We claim the following: for every pair of submodules $(N,P)$ of $M$ such that $N\subset P$, the quotient $P/N$ is semisimple, and $M/N$ is not locally of finite length, there exists submodules $(N',P')$ such that $P\subset P'$, $N'\cap P=N$, $P'/N'$ is semisimple, and such that the canonical injection $P/N\to P'/N'$ is not surjective (i.e., $P'\neq N'+P$), and such that $P'/P$ and $N'/N$ are cyclic. 

Indeed to prove this, we can suppose that $N=0$. By assumption, there exists a cyclic submodule $W$ of $M$ such that $W$ is isomorphic to $A/\PP$; since $\PP$ is non-maximal, we have $W\cap P=0$. Let $N'$ be a maximal proper submodule of $W$ and define $P'=P\oplus W$. So both $N'$ and $P'/P$ are cyclic, $P'/N'\simeq P\oplus W/N'$ is semisimple and strictly contains $P$. This proves the claim.

Now let us prove the result. Suppose that $M$ has no finitely generated submodule such that the quotient is locally of finite length.

Let us define a sequence of pairs of submodules $(N_n,P_n)$, with each $P_n$ a finitely generated submodule containing $N_n$, and the embedding $P_n/N_n\to P_{n+1}/N_{n+1}$ being non-surjective for all $n$. Start with $N_0=P_0=0$. Suppose that it is constructed until $n$. 
By assumption, $M/N_n$ is not locally of finite length. Hence we can apply the above claim, and obtain a pair $(N_{n+1},P_{n+1})$ such that $N_{n+1}/N_n$, $P_{n+1}/P_n$ are cyclic and the inclusion between quotients $P_n/N_n\to P_{n+1}/N_{n+1}$ is non-surjective. Once this sequence is defined, defining $P=\bigcup P_n$ and $N=\bigcup N_n$, we have the subquotient $P/N$ of $M$ semisimple of infinite length, which is an infinite direct sum of nonzero submodules.
\end{proof}

\subsection{Matlis duality for minimax modules}\label{matlis}

Let $R$ be a complete noetherian local ring. Matlis duality is stated in \cite[Theorem 3.2.13]{BH}, summarized in Theorem \ref{matdu}, as a contravariant equivalence of categories between noetherian and artinian modules over $R$. It sounds natural to extend it to the more symmetric context of minimax modules, and this can be done at little cost. In addition, we check that this is the largest setting in which such a duality holds. 

We denote by $E$ an injective hull of the residual field of $R$, and $T(M)=\mathrm{Hom}(M,E)$ is the Matlis dual of $M$. 

\begin{thm}\label{minimatlis}
For every $R$-module $M$, the canonical homomorphism $M\to T(T(M))$ is injective. It is surjective if and only if $M$ is minimax.
\end{thm}

\begin{cor}
The functor $M\mapsto T(M)$ is a contravariant self-equivalence of the category of minimax $R$-modules.
\end{cor}
\begin{proof}
By exactness and since it exchanges noetherian and artinian modules, it maps minimax modules to minimax modules. So the corollary follows from Theorem \ref{minimatlis}.
\end{proof}

\begin{cor}\label{cmatlis}
Let $M$ be a minimax $R$-module. Then the map $\zeta_M:\Sub_R(M)\to\Sub_R(T(M))$, mapping $N$ to $\{f\in T(M):f(N)=0\}$ is a bijection, and is an isomorphism of posets $(\Sub_R(M),\subseteq)\to (\Sub_R(T(M)),\supseteq)$.
\end{cor}

See Lemma \ref{phicont} for a continuity statement (in a particular case).

\begin{proof}[Proof of Theorem \ref{minimatlis}: injectivity]
Let $x$ be a nonzero element of $M$. Then by Corollary \ref{resart}, there exists an artinian quotient $N$ of $M$ in which $x$ has a nonzero image $y$. By Matlis duality for $N$, there exists a homomorphism $N\to E$ that is nonzero on $y$. By composition, we deduce a homomorphism $M\to E$ that is nonzero on $x$; this precisely means that $x$ is not in the kernel of $M\to T(T(M))$ and proves the injectivity.
\end{proof}

\begin{lem}\label{surjt}
The class of $R$-modules $M$ for which $M\to T^2(M)=T(T(M))$ is surjective is stable under taking submodules, quotient modules and extensions.
\end{lem}
\begin{proof}Let $N$ be a submodule of $M$ with $Q=M/N$. Since $E$ is an injective module, $T$ is an exact functor and so is $T^2$. Then we have the commutative diagram, with exact rows
$$\xymatrix{
0\ar[r] & N \ar@{->}[d] \ar[r] & M \ar[d]\ar[r] & Q \ar@{->}[d]\ar[r] &0\\
0 \ar[r] & T^2(N) \ar[r] & T^2(M)\ar[r]& T^2(Q) \ar[r] &0}.$$

1) Let us start with extensions. If $N\to T^2(N)$ and $Q\to T^2(Q)$ are both surjective, it follows from the above diagram that the middle downwards arrow $M\to T^2(M)$ is also surjective.

2) Now suppose that $M\to T^2(M)$ is surjective. Then the composite map $M\to Q\to T^2(Q)$ is surjective, so $Q\to T^2(Q)$ is surjective as well.

Let us show that $N\to T^2(N)$ is surjective as well. Let $f$ be an element of $T(T(N))$. Its image $f'$ in $T(T(M)$ is the homomorphism $T(M)\to E$ mapping $\phi$ to $f(\phi\circ i)$, where $i$ is the injection $N\to M$. The surjectivity for $f$ says that there exists $m\in M$ such that $f'(\phi)=\phi(m)$ for all $m\in M$. 
We have to prove that $m$ belongs to the image of $i$. Assuming otherwise, the injectivity of $(M/N)\to T^2(M/N)$ implies that there exists $\phi\in T(M)$ that vanishes on $N$ and such that $\phi(m)\neq 0$. Then $f'(\phi)=\phi(m)\neq 0$ on the one hand, and $f'(\phi)=f(\phi\circ i)=0$ on the other hand, a contradiction.
\end{proof}

\begin{proof}[Proof of the surjectivity statements in Theorem \ref{minimatlis}]
If $M$ is minimax, we use the extension stability of Lemma \ref{surjt}, and the surjectivity of $M\to T^2(M)$ follows from the noetherian and artinian cases, established in \cite[Theorem 3.2.13]{BH}. 

Conversely, if $M$ is not minimax, then $M$ admits, by Proposition \ref{exiqtsemi}, an infinite direct sum of nonzero modules as subquotient, and hence admits $V=\mathbf{k}^{(\mathbf{N})}$ as subquotient, where $\mathbf{k}$ is the residual field. Then $T(V)\simeq \mathbf{k}^{\mathbf{N}}\simeq \mathbf{k}^{(c)}$ (where we write $c=2^{\aleph_0}$), so $T^2(V)\simeq \mathbf{k}^{c}\simeq \mathbf{k}^{\left(2^{c}\right)}$ is not even isomorphic to $V$. It follows from Lemma \ref{surjt} (stability under quotients and submodules of the surjectivity property) that $M\to T^2(M)$ is not surjective.
\end{proof}

\begin{proof}[Proof of Corollary \ref{cmatlis}]
By Matlis duality, we identify $T(T(M))$ with $M$, under this identification, for $N\in\Sub_R(M)$, we have
\[\zeta_{T(M)}\circ\zeta_M(N)=\{x\in M:\forall f\in T(M),\;f(N)=0\Rightarrow f(x)=0\};\]
it clearly contains $N$; conversely if $x\notin N$ then since $x$ is not in the kernel of $M\to T(T(M/N))$, we see that $x\notin \zeta_{T(M)}\circ\zeta_M(N)$. Hence $\zeta_{T(M)}\circ\zeta_M$ is the identity of $\Sub_R(M)$. Applying this to $T(M)$, $\zeta_{M}\circ\zeta_{T(M)}$ is the identity of $\Sub_R(T(M))$. So $\zeta_M$ is a bijection, and it is clearly order-reversing as well as its inverse $\zeta_{T(M)}$.
\end{proof}

\subsection{Property {\LL} and Loewy dimension}

\begin{proof}[Proof of Proposition \ref{l2intro}]
(\ref{l2_2})$\Rightarrow$(\ref{l2_1}) If $\textnormal{Sub}_A(M)$ contains a chain isomorphic to $\omega^2$ and $M$ has a (finite) composition series with subquotients $M_i$, then for at least one $i$, $\textnormal{Sub}_A(M_i)$ also contains a chain isomorphic to $\omega^2$. Since clearly the modules in (\ref{l2_2}) have no such chains, we deduce that (\ref{l2_2}) implies (\ref{l2_1}).

(\ref{l2_1})$\Rightarrow$(\ref{l2_3}) is immediate.

(\ref{l2_3})$\Rightarrow$(\ref{l2_2}): first, assuming (\ref{l2_3}), we know that $M$ is minimax by Theorem \ref{minidir} (which was proved in \S\ref{smx}). Thus define $M_0$ as a finitely generated submodule such that $M/M_0$ is artinian. We define, by induction, a sequence of submodules $M_i$ (containing $M_0$) as follows. Assuming that $M_i$ is defined, if $M/M_i$ has finite length, define $M_{i+1}=M$; otherwise, $M_{i+1}$ is chosen to be minimal among submodules containing $M_i$ such that $M_{i+1}/M_i$ has infinite length; since $M/M_i$ is artinian of infinite length, this does exist. If $M_i\neq M$ for all $i$, each $M_{i+1}/M_i$ has infinite length and from this sequence we can interpolate to obtain an embedding of $\omega^2$ in the poset $\textnormal{Sub}_A(M)$, a contradiction. So $M=M_i$ for some $i$, which is precisely what we need to obtain (\ref{l2_3}).
\end{proof}

We can now formulate an additional characterization.

\begin{prop}\label{addl2}
Let $M$ be an $A$-module. Then $A$ satisfies (\ref{l2_2}) of Proposition \ref{l2intro} if and only if it satisfies (iv): $M$ is minimax of Loewy dimension $\le 1$.
\end{prop}
\begin{proof}
It is clear from the definition that having Loewy dimension $\le\kappa$ is stable under taking extensions. Hence, to prove (\ref{l2_2})$\Rightarrow$(iv), it is enough to show that whenever $M$ is either finitely generated or has all its proper submodules of finite length, then $M$ is minimax of Loewy dimension $\le 1$. That $M$ is then minimax is clear (since it is then finitely generated or artinian). If $M$ is finitely generated, then its Loewy dimension is $\le 0$. The remaining case is when $M$ has infinite length with all its proper submodules of finite length, and in particular is artinian. It follows that $M$ does not split as a direct product, and hence has a single associated ideal $\mathfrak{M}$. The Matlis dual of $M$ is then a finitely generated $\widehat{A_\mathfrak{M}}$-module of infinite length, with all its proper quotients of finite length, and then has Krull dimension 1. So the Loewy dimension of $M$ is 1.

Conversely, let us prove (iv)$\Rightarrow$(\ref{l2_2}). Suppose that $M$ is minimax of Loewy dimension $\le 1$, and let us show that it satisfies (\ref{l2_2}). Since (\ref{l2_2}) is stable under taking extension, we can suppose that $M$ is either finitely generated (this case being clear), or is artinian with a single associated ideal $\mathfrak{M}$, as we now assume. The Matlis dual of $M$, as an $\widehat{A_\mathfrak{M}}$-module, is finitely generated of Krull dimension $\le 1$, and hence has a finite composition series with each successive quotient $Q_i$ being either of finite length or isomorphic to some module $\widehat{A_\mathfrak{M}}/\mathcal{P}_i$ for some prime ideal $\mathcal{P}_i$ of $\widehat{A_\mathfrak{M}}$ of coheight 1. In particular, every proper quotient of $Q_i$ has finite length. This yields a composition series of $M$, as an $A$-module, in which each successive quotient has all its proper submodules of finite length.
\end{proof}

\begin{cor}\label{l2cara}
Let $A$ be a noetherian ring and $M$ an $A$-module. Then $M$ satisfies Property {$(\mathsf{L}_2)$} if and only if $M$ has Loewy dimension $\le 1$.
\end{cor}
\begin{proof}
Since each condition only depends on artinian quotients of $M$, this follows from the case when $M$ is artinian, which follows from Propositions \ref{l2intro} and \ref{addl2} (namely, the resulting equivalence between (\ref{l2_1}) and (iv)).
\end{proof}

\begin{rem}\label{noquo2}
Let $M$ fail to satisfy Property {\LL}. By definition, it has an artinian quotient $N$ without Property {\LL}. Using an associated ideal and Matlis duality, this means that there exists $d\ge 2$ and a quotient $P$ of $N$ with the following property: $P$ has Loewy dimension $d$ as well as all its nonzero quotients, while all proper submodules of $P$ have Loewy dimension $<d$. In particular, if $d\ge 3$, $P$ has no quotient of Loewy dimension 2. \end{rem}

\subsection{Ordinal length $\omega$}\label{s_lomega}

\begin{lem}\label{lomega}
Let $M$ be a finitely generated $A$-module. Equivalences:
\begin{enumerate}[(i)]
\item\label{ii1} the ordinal length of $M$ is equal to $\omega$;
\item\label{ii2} $M$ has a unique associated ideal $\PP$, which has coheight one, and $M$ is a torsion-free $A/\PP$-module of rank one.
\end{enumerate}
\end{lem}
\begin{proof}
Suppose (\ref{ii2}). Clearly $\ell(M)\ge\omega$. If $N$ is a nonzero submodule, then since $M$ is torsion-free of rank 1, it follows that every associated ideal of $M/N$ strictly contains $\PP$, so $M/N$ has finite length. So $\ell(M)\le\omega$ and thus $\ell(M)=\omega$.

Conversely suppose $\ell(M)=\omega$. Since $M$ has infinite length, it has a non-maximal associated prime ideal $\PP$. Since every proper quotient of $M$ has finite length, it is clear that $\PP$ has coheight one. Moreover, if by contradiction $\mathcal{Q}$ is another associated prime ideal, then $M$ contains a copy $N$ of $A/\mathcal{Q}$ and $M/N$ is a proper quotient of $M$ containing a copy of $A/\PP$ and thus of infinite length, contradicting that $\ell(M)=\omega$. So $\PP$ is the only associated ideal.

Now $A/\PP$ embeds into $M$, so taking the tensor product by $A_\PP$, we see that $A_\PP/\PP A_\PP$ embeds into $M_\PP=M\otimes A_\PP$, thus $M_\PP$ is nonzero, showing by Nakayama's Lemma that $M_\PP/\PP A_\PP M_\PP\neq 0$. Since $M_\PP/\PP A_\PP M_\PP=(M/\PP M)\otimes A_\PP$, we deduce that $M/\PP M$ has an associated ideal contained in $\PP$, and thus has infinite length. Since $\ell(M)=\omega$, it follows that $\PP M=0$, i.e.\ $M$ is an $A/\PP$-module. Given again that $\ell(M)=\omega$, it is now immediate that it is torsion-free of rank one.
\end{proof}

\begin{lem}\label{plusone}
Let $M$ be a finitely generated $A$-module. Equivalences:
\begin{enumerate}[(i)]
\item\label{iii1} the ordinal length of $M$ is equal to $\omega+1$;
\item\label{iii2} $M$ has a simple submodule $N$ such that $\ell(M/N)=\omega$.
\end{enumerate}
Moreover, in (\ref{iii2}), $N$ is equal to the socle of $M$ (and hence is unique).
\end{lem}
\begin{proof}
Note that this statement and the argument below is very general (e.g., for a noetherian module over an arbitrary associative ring).

Suppose (\ref{iii2}). Since $\ell(M/N)\ge\omega$ and $N\neq 0$, we obtain $\ell(M)\ge\omega+1$. If $P$ is a submodule of $M$ and is not contained in $N$, we have $0\to N/(N\cap P)\to M/P\to M/(N+P)\to 0$, so $\ell(M/P)<\omega$, while $\ell(M/N)=\omega$, so $\ell(M/N)\le\omega+1$.

Suppose (\ref{iii1}). There exists $N$ such that $\ell(M/N)=\omega$. Then $N\neq 0$. If $N$ is not simple, say has the nonzero proper submodule $N'$, then $\ell(M)>\ell(M/N')>\ell(M/N)$ and hence $\ell(M)\ge\omega+2$, a contradiction, so $N$ is simple. 

If $N'\neq N$ is another simple submodule, then the image of $N'$ in $M/N$ is nonzero, so $\ell(M/(N\oplus N'))<\ell(M/N)=\omega$. So $M/(N\oplus N')$ has finite length, and $N\oplus N'$ has length 2; hence $M$ has finite length, a contradiction.
\end{proof}

\subsection{Artinian length}

Let $M$ be an artinian $A$-module (for the definition, $A$ need not be commutative). The artinian length of $M$ is defined inductively as $\mathcal{L}(M)=\sup\{\mathcal{L}(N)+1\}$, where $N$ ranges over proper submodules of $M$ (with $\sup\emptyset=0$). 

If finite, this coincides with the usual length. 
Note that, by definition, $\mathcal{L}(M)=\omega$ precisely means that $M$ has infinite length while all its proper submodules have finite length.

\begin{prop}
Let $A$ be a ring (possibly non-commutative) and $M$ an artinian $A$-module. Then
\begin{enumerate}
\item\label{maxchain} every chain of submodules of $M$ has ordinal type $\le\mathcal{L}(M)$;
\item\label{chainmax} conversely, if $\mathcal{L}(M)$ is countable, then there exists a chain of submodules of $M$ of ordinal type $\mathcal{L}(M)$.
\end{enumerate}
\end{prop}
\begin{proof}
The first part, (\ref{maxchain}), is proved by a straightforward induction. The more subtle part, (\ref{chainmax}), is an immediate adaptation of the dual statement for ordinal length \cite[Prop.\ 2.12]{Gull}.
\end{proof}

\begin{prop}
Let $A$ be a (commutative) noetherian ring and $M$ an artinian $A$-module.
We have $\omega^n\le \mathcal{L}(M)<\omega^{n+1}$, where $n$ is the Loewy dimension of $M$.
\end{prop}
\begin{proof}
If $M$ has a single associated prime ideal, this follows from the corresponding result of Gulliksen \cite[Theorem 2.3]{Gull} for Krull dimension and descending chains, and Corollary \ref{cmatlis}. In general $M$ is a finite product of such modules and the result follows easily.
\end{proof}

\begin{prop}
Let $A$ be a (commutative) noetherian ring and $M$ an $A$-module. The following are equivalent:
\begin{enumerate}[(i)]
\item\label{PPdef} $M\simeq P\oplus P$ for some artinian $A$-module $P$ of length $\omega$;
\item\label{PPcar} $M$ is artinian of artinian length $\le \omega+\omega$ and with 3 submodules $M_1,M_2,M_3$ each of artinian length $\ge\omega$ such that $M_i\cap M_j=\{0\}$ for all $1\le i<j\le 3$.
\end{enumerate}
\end{prop}
\begin{proof}
The forward direction is clear. Conversely, assume that (\ref{PPcar}) holds. Passing to submodules, we can assume that the artinian length of $M_i$ is exactly $\omega$ for all $i$. Then $M_1\oplus M_2$ has artinian length $\omega+\omega$, and hence is equal to $M$. For $i=1,2$, the projection of $M=M_1\oplus M_2\to M_i$ is injective on $M_3$, and hence its image has artinian length $\omega$; thus by definition of artinian length of $M_i$, this projection is surjective in restriction to $M_3$. Hence $M_3$ is the graph of an isomorphism $M_1\to M_2$, proving (\ref{PPdef}).
\end{proof}

This proposition allows to characterize Property {\LLL} purely in terms of the poset of submodules.
For convenience, we use the following abbreviation. Consider a poset $(X,\le)$ with a minimal element $o$, and denote $[a,b]=\{c:a\le c\le b\}$. We say that $X$ satisfies Property $(\sharp)$ if $X$ includes no chain of ordinal type $>\omega+\omega$, and there are $x_1,x_2,x_3\in X$ such that $[o,x_i]\cap [o,x_j]=\{o\}$ for all $1\le i<j\le 3$ such that $X_i=\{x\in X:x\le x_i\}$ is non-noetherian for all $i$.  

\begin{cor}\label{posetlll}
Let $A$ be a (commutative) noetherian ring and $M$ an $A$-module. Then 
\begin{enumerate}
\item $M$ is isomorphic to $P\oplus P$ for some artinian $A$-module $P$ of artinian length $\omega$ if and only if $X=\Sub_A(M)$ is artinian and satisfies Property $(\sharp)$.
\item\label{failslllposet} $M$ fails to satisfy Property {\LLL} if and only if there exist $U\subset V$ in $\Sub_A(M)$ such that the poset $\{W:U\subset W\subset V\}$ satisfies Property $(\sharp)$.\qed
\end{enumerate}
\end{cor}

\section{Finitely generated modules over complete local rings}\label{ficom}

\subsection{Cardinality of complete local rings}

We need the following easy observation.

\begin{lem}\label{ccom}
Let $(R,\MM)$ be a complete local ring with residual field $\mathbf{k}=R/\MM$ of cardinality $\alpha$. Then the cardinality of $R$ is equal to
\begin{itemize}
\item $\alpha^{\aleph_0}$ if $R$ is non-artinian;
\item $\alpha^{\ell(R)}$ if $R$ is artinian (of finite length $\ell(R)$ as $R$-module); in particular in this case, the cardinality of $R$ equals $\alpha$ if $\mathbf{k}$ (or equivalently $R$) is infinite;
\end{itemize}
\end{lem}
\begin{proof}
If $R$ is artinian, then $R$ has a composition series, as an $R$-module, with subfactors $\mathbf{k}$ so the last assertion follows. In general, $R/\MM^n$ is artinian and the above applies. Since $R$ embeds as a subring into $\prod_n R/\MM^n$, we deduce that the cardinality of $R$ is at most $\alpha^{\aleph_0}$ (noting if $\alpha$ is finite that $\alpha^{\aleph_0}=\aleph_0^{\aleph_0}$). Let us check that this is an equality if $R$ is not artinian. So by assumption $R/\MM^{n+1}\to R/\MM^n$ is a bijection for no $n$. If $\alpha$ is infinite, every element has exactly $\alpha$ preimages and picking, for every $n$ and every element $x\in R/\MM^n$, a bijection $v_x$ from $\alpha$ to the set of preimages of $x$, we easily deduce an injection of $\alpha^{\aleph_0}$ into $R$, mapping any $(u_0,\dots)\in\alpha^{\aleph_0}$ to the sequence $(x_0,\dots)$ of $R$, where $x_0=0$ and $x_{n+1}=v_{u_n}(x_n)$, which is an element of the projective limit. If $\alpha$ is finite, a similar argument holds; note that in this case $(R,+)$ is a profinite group obtained as an inverse limit of a sequence of finite groups; since it is infinite, it is homeomorphic to a Cantor set.
\end{proof}

\subsection{Topology on submodules}\label{toposub}

Let $(R,\MM)$ be a complete local ring and $M$ a finitely generated $R$-module. There is a natural topology on the set $\Sub_R(M)$ of submodules of $M$, usually strictly finer than the topology induced by inclusion in $2^M$.

Namely, it is defined by the ultrametric distance
$$d(N,N')=\exp\left(-\sup\{n\ge 0:\;\;N\subset N'+\MM^nM\textnormal{ and }N'\subset N+\MM^nM\}\,\right).$$
This is indeed a distance (and not only a semi-distance), because submodules are closed (Proposition \ref{isclosed}).

\begin{lem}\label{phicont}
Let $R$ be a complete noetherian local ring with residual field $\mathbf{k}$. Fix an injective hull $E$ of $\mathbf{k}$ as an $R$-module, defining Matlis duality $M\mapsto T(M)=\Hom(M,E)$. 
Let $M$ be a finitely generated $R$-module. Endow
\begin{itemize}
\item $\Sub(M)$ with the topology introduced in \ref{toposub};
\item $\Sub(T(M))$ with the topology defined by inclusion into $2^{T(M)}$.
\end{itemize}

The the resulting bijection $\Sub(M)\to\Sub(T(M))$ arising from Matlis duality (mapping $N$ to the ``orthogonal" $\{f\in T(M):f(N)=0\}$, see Corollary \ref{cmatlis}) is continuous.
\end{lem}

\begin{proof}
Observe that the decreasing sequence of ideals $\MM^n$ corresponds to an increasing sequence of submodules $E_n$ of $E$, with $E=\bigcup E_n$, and then remark that 
$d(I,J)\le\exp(-n)$ implies that if $\phi(I)\cap E_n=\phi(J)\cap E_n$. 
\end{proof}

Even if we will not use it, let us mention that the reciprocal bijection is usually not continuous.

\begin{prop}Under the assumptions of Lemma \ref{phicont},
the reciprocal bijection $\Sub(T(M))\to \Sub(M)$ is continuous if and only if $\mathbf{k}$ is finite or $M$ is uniserial.
\end{prop}
\begin{proof}
When $\mathbf{k}$ is finite, $M$ is compact metrizable, the topology on $\Sub(M)$ is induced by the Hausdorff distance on compact subsets and thus is compact, and hence the continuous bijection has to be a homeomorphism.

When $M$ is uniserial, it also follows that $\Sub(M)$ is compact: indeed, either $M$ has finite length and in this case $\Sub(M)$ is finite, or $M$ has infinite length and is isomorphic to a discrete valuation ring of the form $R/\PP$ (this is well-known, see Lemma \ref{dome} if necessary). Then $\Sub(M)$ is an infinite chain, namely a discrete descending sequence and $\{0\}$, which is indeed the limit of this sequence in the topology of $\Sub(M)$. 

Now assume that $\mathbf{k}$ is infinite and $M$ is not uniserial. Then $M$  has submodules $P\subset N$ with $N/P\simeq (R/\MM)^2$. Then the set $X$ of submodules containing $P$ and containing $N$ is a closed subset of $M$.
There exists $n$ such that $\MM^nM\cap N\subset P$. It follows that $X$ is discrete. Since $\mathbf{k}$ is infinite, $X$ is infinite. Hence $X$ is not compact, so $\Sub(M)$ is not compact. Since $\Sub(T(M))$ is compact, we deduce that they are not homeomorphic. 
\end{proof}

\subsection{Krull dimension at least two}\label{remi_k2}

\begin{thm}\label{grande}
Let $(R,\MM)$ be a complete noetherian local ring of Krull dimension at least 2. Then $R$ has uncountably many ideals. More precisely, it has uncountably many prime ideals of height 1.
\end{thm}
\begin{proof}
Endow $R$ with the $\MM$-adic topology. Note that this is a Baire space (homeomorphic to $\mathbf{k}^\mathbf{Z}$), and that ideals are all closed (Proposition \ref{isclosed}). 

The Hauptidealsatz says that for every $x\in\MM$, every minimal prime ideal among those containing $Rx$ has height $\le 1$. Hence $\MM$ is the union of all prime ideals of height $\le 1$. Now every non-maximal prime ideal, and in particular every prime ideal of height $\le 1$, has empty interior (because if a prime ideal contains $\MM^n$ for some $n$, it should contain $\MM$). By the Baire category theorem, $\MM$ cannot be covered by countably many closed subsets with empty interior. So there are uncountably many prime ideals of height $\le 1$. Since there are finitely many prime ideals of height 0, the result follows. 
\end{proof}

\begin{prop}\label{grandes}
Let $R$ be a noetherian local ring of Krull dimension at least two. Then the set of ideals $\mathcal{I}(R)=\Sub_R(R)$, with the topology introduced in \S\ref{toposub}, is not scattered, and more precisely its subset of principal proper ideals is a nonempty perfect subset.
\end{prop}
\begin{proof}
Since $\MM$ has height at least 2, so does $\MM^n$ for all $n\ge 1$ as well as any ideal containing $\MM^n$ for some $n\ge 1$. By Lemma \ref{k2}, for every $x\in\MM$ and $n\ge 1$, there exists $\varepsilon_n\in\MM^n\smallsetminus Rx$. 
 
Clearly, for the given topology on $\mathcal{I}(R)$, the sequence $(R(x+\eps_n))$ tends to $Rx$ and $R(x+\eps_n)\neq Rx$. In particular, the set $\mathcal{P}\subset\mathcal{I}(R)$ consisting of $Rx$, for $x\in\MM$, is a nonempty perfect set (i.e.\ without isolated points).
\end{proof}

This yields, with a somewhat more complicated proof (as we use the material of \S\ref{toposub}), an improvement of the first part of Theorem \ref{grande}. 

\begin{cor}
Let $R$ be a complete noetherian local ring of Krull dimension at least two and $E$ an injective hull of its residual field. Then the compact set $\Sub(E)$ of submodules of $E$, endowed with the topology induced by inclusion in $2^E$, is not scattered and has cardinality $\ge 2^{\aleph_0}$. In particular, $R$ has $\ge 2^{\aleph_0}$ ideals. 
\end{cor}
\begin{proof}
By Proposition \ref{grandes}, $\mathcal{I}(R)$ contains a nonempty perfect subset $\mathcal{P}$. Denote by $\phi:\mathcal{I}(R)\to\Sub(E)$ the bijection induced by Matlis duality; it is continuous by Lemma \ref{phicont}. Hence $\overline{\phi(\mathcal{P})}$ is a nonempty compact perfect subset of $\Sub(E)$. In particular, the latter is not scattered, and by Lemma \ref{ontocantor} it has at least $2^{\aleph_0}$ elements. 
In turn, since $\phi$ is a bijection, it follows that $\mathcal{I}(R)$ has cardinality $\ge 2^{\aleph_0}$. 
\end{proof}

We now improve Theorem \ref{grande} to obtain the exact cardinality, at the cost of a more involved proof. 

\begin{thm}\label{grandeb}
Let $R$ be a complete local ring of Krull dimension at least 2, and residual field $\mathbf{k}$ of cardinal $\alpha$. Then $R$ has exactly $\alpha^{\aleph_0}$ ideals.
\end{thm}
\begin{proof}
Since $R$ is noetherian and has cardinality $\alpha^{\aleph_0}$ by Lemma \ref{ccom}, this is an obvious upper bound.

For $x,y\in R$ and $n\ge 1$, we say that $x\sim_n y$ if $Rx+\MM^n=Ry+\MM^n$. This is obviously an equivalence relation. We claim that $x\sim_n y$ if and only if there exists $t\in R\smallsetminus\MM$ and $z\in \MM^n$ such that $x=ty+z$. The ``if" part being trivial, assume that $x\sim_n y$. Then $x$ and $y$ generate the same ideal in the quotient $R/\MM^n$. This means that modulo $\MM^n$, we can write $x=ty$ with $t\notin\MM$. Lifting this to $R$, we obtain $x=ty+z$ with $z\in\MM^n$.

Define $\sim_\infty=\bigcap\sim_n$. This is a decreasing union, so $\sim_\infty$ is an equivalence relation as well. Clearly if $Rx=Ry$ then $x\sim_\infty y$. 

Define a rooted tree $\mathcal{T}$ as follows. For $n\ge 1$, the $n$th level is the quotient $\mathcal{T}_n$ of $\MM$ by the equivalence relation $\sim_n$. For $n\ge 2$, each $\sim_n$-class is contained in a unique $\sim_{n-1}$-class; the corresponding vertex of level $n$ is thus connected to the corresponding vertex of level $n-1$. (Note that the root is of level one).

Now define $\mathcal{T}_\infty$ as the set of geodesic rays of $\mathcal{T}$ (emanating from the root). Namely, a geodesic ray is a sequence $r=(r_n)_{n\ge 1}$ with $r_n\in\mathcal{T}_n$ and $r_{n+1}\subset r_n$ for all $n\ge 1$. 

Let $\mathcal{T}_{(\infty)}$ be the quotient of $\MM$ by the equivalence relation $\simeq_\infty$. There is a canonical map $\mathcal{T}_{(\infty)}\to\mathcal{T}_n$, mapping a $\sim_\infty$-class to the unique $\sim_n$-class containing it; together they define a canonical map $\mathcal{T}_{(\infty)}\to\mathcal{T}_\infty$, which is immediately seen to be injective.

Let us show that this map is bijective. We have to construct the inverse map. Namely, given a ray $(r_n)$ as above, we set $r_\infty=\bigcap r_n$. Observe that if $r_\infty$ is not empty, then it is a $\sim_\infty$-class. Indeed, clearly for all $x,y\in r_\infty$ we have $x\sim_\infty y$; conversely if $x\in r_\infty$ and $y\sim_\infty x$, then $x\in r_n$ for all $n$, and $y\sim_n x$ so $y\in r_n$ as well, so $y\in r_\infty$. Moreover, this is necessarily the preimage of $r_\infty$ in the previous injection.

Let us now check that $r_\infty$ is not empty; here we shall use the fact that $R$ is complete.
Let $x_n$ be a representative of the $\sim_n$-class $r_n$.
Since $(r_n)$ is a ray, $x_n\sim_n x_{n+1}$ for all $n$. This means that we can write $x_{n+1}=t_nx_n+z_n$ with $t_n\in R\smallsetminus\MM$ and $z_n\in\MM^n$. Define $y_n=\left(\prod_{1\le j\le n-1}t_j\right)^{-1}x_n$ and $\zeta_n=\left(\prod_{1\le j\le n}t_j\right)^{-1}z_n$, so $\zeta_n\in\MM^n$. Then $(y_n)$ is another representative of $(r_n)$, and we have, for all $n$, $y_{n+1}=y_n+\zeta_n$. This means that the sequence $(y_n)$ is convergent in $R$ (endowed with its inverse limit topology), to a limit $y$, characterized by the fact that $y-y_n\in\MM^n$ for all $n$. Thus $y\sim_n y_n$, i.e.\ $y\in r_n$ for all $n$, so $y\in r_\infty$.

What we finally have to check is that the tree is everywhere branched of degree $\alpha$. Let $r$ be a vertex of level $n$, and $x$ a representative of $r$. We have to check that $r$ has exactly $\alpha$ successors, i.e.\ that there exist $\alpha$ distinct elements $y$ that are $\sim_n$-equivalent to $x$ and pairwise not $\sim_{n+1}$-equivalent. We consider two cases.

\begin{itemize}
\item Suppose $x\in\MM^n$. All $z\in\MM^n$ are $\sim_n$-equivalent to $x$. If two such elements $z,z'$ are $\sim_{n+1}$-equivalent, then for some $t\in R\smallsetminus\MM$ we have $z=tz'+o(\MM^n)$. Therefore to conclude we have to check that the quotient of $\MM^n/\MM^{n+1}$ by the action by multiplication of $\mathbf{k}=R/\MM$ has $\alpha$ orbits. This is clearly equivalent to say that the dimension of $\MM^n/\MM^{n+1}$ as $\mathbf{k}$-vector space is at least two, which is an immediate consequence of Lemma \ref{k2}.

\item Suppose $x\notin\MM^n$. Then the ideal $I=\{a\in R:ax\in\MM^n\}$ is contained in $\MM$. Let us consider elements of the form $x+z$ with $z\in\MM^n$; they are clearly $\sim_n$-equivalent to $x$.
Suppose that two such elements, $x+z$ and $x+z'$, are $\sim_{n+1}$-equivalent. For convenience, we use a Landau notation and write any element of $\MM^{n+1}$ as $o(\MM^n)$.  Then there exists $\lambda\in R\smallsetminus\MM$ such that $x+z=t(x+z')+o(\MM^{n})$, so $(1-t)x=tz'-z+o(\MM^{n})$. Therefore $(1-t)\in I$, so $(1-t)\in\MM$. Therefore $tz'-z'=o(\MM^n)$, and thus
$(1-t)x=z'-z+o(\MM^n)$. Define $J_n=(Rx\cap\MM^n)+\MM^{n+1}\subset\MM^n$. We just proved that if $x+z$ and $x+z'$ are $\sim_{n+1}$-equivalent, then $z-z'\in J_n$. Now observe that $\MM^n/J_n$ is a $\mathbf{k}$-vector space; it has cardinality $\alpha$ unless $\MM^n\subset J_n$, but this is absurd: indeed this implies that $\MM^n\subset \MM^{n+1}+Rx$, which is discarded by Lemma \ref{k2}.\qedhere
\end{itemize}
\end{proof}

\begin{cor}\label{crande}
Let $R$ be a complete noetherian local ring, with residual field $\mathbf{k}$ of cardinality $\alpha$. Let $M$ a finitely generated $R$-module of Krull dimension at least two. Then $M$ has exactly $\alpha^{\aleph_0}$ submodules. In particular if $\mathbf{k}$ is countable then $M$ has exactly $2^{\aleph_0}$ submodules. 
\end{cor}
\begin{proof}
It is an upper bound by Lemma \ref{ccom}. Conversely $M$ has an associated ideal $\PP$ of coheight at least two, i.e.\ $R/\PP$ embeds into $M$ as a submodule and Theorem \ref{grandeb} (or Theorem \ref{grande} when $\mathbf{k}$ is countable) applies.
\end{proof}

\subsection{Krull dimension one}
Let $(R,\MM)$ be a complete local ring. The ring $R$ being endowed with the $\MM$-adic topology, and the set of submodules being topologized as in \S\ref{toposub}, the following lemma is immediate.
\begin{lem}\label{1b}
The map $R\to\Sub_R(R^2)$, mapping $b$ to the $R$-submodule generated by $(1,b)$, is injective and continuous. In particular, $\Sub_R(R^2)$ is non-scattered as soon as $R$ is non-artinian.\qed
\end{lem}
 
We deduce:

\begin{prop}\label{carre0}
Let $(R,\MM)$ be a complete noetherian local ring with residual field $\mathbf{k}$ of cardinality $\alpha$. Let $M$ be an $R$-module of Krull dimension $\le 1$. Suppose that for some non-maximal prime ideal $\PP$, $M$ possesses $(R/\PP)^2$ as a subquotient.

Then $M$ has $\ge$ $\alpha^{\aleph_0}$ submodules (i.e.\ $2^{\aleph_0}$ if $\alpha$ is countable), with equality if $M$ is finitely generated.
\end{prop}
\begin{proof}
For the inequality $\ge$, Lemma \ref{1b} shows that the cardinal of $\Sub(M)$ is greater or equal than that of $R$, which is given by Lemma \ref{ccom}. The inequality $\le$, when $M$ is finitely generated, is immediate because a submodule is determined by a finite generating family, so the cardinal of $\Sub(M)$ is bounded above by the cardinal of $R^n$ for some $n$, itself bounded above by $\max(\aleph_0,\#(R))$.
\end{proof}

Let us now prove the converse of Proposition \ref{carre0}. Recall that a module is meager if it does not admit any subquotient isomorphic to $N^2$ for any simple module $N$.

\begin{prop}\label{carre}  
Let $(R,\MM)$ be a complete noetherian local ring with residual field $\mathbf{k}$ of cardinality $\alpha$. 
Let $M$ be a finitely generated $R$-module of Krull dimension $\le 1$.

Suppose that for every non-maximal prime ideal $\PP$, $M$ does not possess $(R/\PP)^2$ as a subquotient. Then $M$ admits at most $\alpha$ submodules. More precisely

\begin{itemize}
\item if $M$ is not meager and infinite, then it has exactly $\max(\alpha,\aleph_0)$ submodules;
\item if $M$ is meager, then it has $\aleph_0$ or finitely many submodules according to whether $M$ has infinite or finite length.
\item if $M$ is finite (meager or not) then it has finitely many submodules.
\end{itemize}
\end{prop}

\begin{proof}
Let us first prove that $M$ has at most $\alpha'=\max(\alpha,\aleph_0)$ submodules. We argue by induction on the ordinal length $\ell(M)$ (see Section \ref{remi}). Note that as an immediate consequence of the definition of ordinal length, we have $\ell(M)<\omega$ if and only if $M$ has finite length in the usual sense, i.e.\ when $M$ is both noetherian and artinian.

If $\ell(M)<\omega$, then $M$ has cardinality $\le\alpha'$ (and is noetherian) and the conclusion is obvious. Now assume that $\ell(M)<\omega\cdot 2$. This means that for every submodule $N$ of $M$ either $\ell(N)<\omega$ or $\ell(M/N)<\omega$. Now on the one hand the number of submodules of finite length is at most $\alpha'$ (because they are all contained in the maximal finite length submodule), and on the other hand every finitely generated $R$-module has at most $\alpha'$ submodules of given finite colength (the colength of $N$ is by definition $\ell(M/N)$): indeed, the number of submodules $N$ with $M/N\simeq \mathbf{k}$ is controlled by the number of homomorphisms of $M$ into $\mathbf{k}$, which has cardinality at most $\alpha'$ since it is a finitely generated $\mathbf{k}$-module.

Now suppose $\ell(M)\ge\omega\cdot 2$ and assume the condition on subquotients is satisfied, and that the assertion has been proved for all modules of lesser ordinal length. If $\PP\in\Ass_R(M)$, denote by $T_\PP$ the $\PP$-torsion in $M$, i.e.\ the set of elements killed by $\PP$. If $\PP$ is maximal, then $\ell(T_\PP)<\omega$, and otherwise, the assumptions that $(R/\PP)^2$ is not a subquotient of $M$ and and that the Krull dimension of $M$ is at most 1 imply that $\ell(T_\PP)<\omega\cdot 2$. Therefore in all cases, $T_\PP$ has at most $\alpha'$ submodules by the previous case.

Now let $N$ be a non-zero submodule of $M$. Then $N$ has one associated ideal, so has non-empty intersection $N'$ with at least one of the $T_\PP$. There are at most $\alpha'$ possibilities for $N'$ (by the previous argument), and at most $\alpha'$ many possibilities for the submodule $N/N'$ of $M/N'$ (by induction hypothesis). So the second assertion is proved.

To conclude, observe that
\begin{itemize}
\item If $\alpha$ is infinite, then $\mathbf{k}^2$ admits $\alpha$ submodules;
\item If the length of $M$ is infinite, then the number of submodules of $M$ is infinite.
\item If $\mathbf{k}$ is finite and the length of $M$ is finite, then $M$ is finite and thus has only finitely many submodules.
\end{itemize}
In view of this, the only last verification is that if $\mathbf{k}$ is infinite and $M$ does not admit $\mathbf{k}^2$ as a subquotient, then the number of submodules is finite or $\aleph_0$ according to whether $M$ has finite length. Observe that in this case, the set of submodules of $M$ is totally ordered, consisting of a finite or infinite descending sequence, and the result follows.\end{proof}

\subsection{Non-redundancy of the conditions}
Let us give a large family of examples indicating that Theorem \ref{grande} cannot be deduced formally from Proposition \ref{carre0}. In other words, it shows that Condition {\KK} does not always follow from Condition {\KKK} in Theorem \ref{thmB}.

\begin{prop}\label{deuxd}
Let $A$ be a noetherian unique factorization domain of Krull dimension 2. Then for every non-maximal ideal $\PP$ of $A$, the module $(A/\PP)^2$ is not a subquotient of $A$.
\end{prop}
\begin{proof}

Let us show more generally that if a noetherian ring $A$ is non-singular in codimension $\le 1$ (i.e.\ for every prime ideal of height $\le 1$, the local ring $A_\PP$ is regular; this holds if $A$ is a UFD or more generally a normal domain), then for every prime ideal of height $\le 1$ there is no $A$-module embedding of $(A/\PP)^2$ into any quotient $A/I$ of $A$. 

Indeed, assume there is such an embedding. Taking the tensor product with $A_\PP$, we deduce an embedding of $(A_\PP/\PP A_\PP)^2$ into $A_\PP/IA_\PP$. If $\PP$ has height zero, then $B=A_\PP$ is a field and $B^2$ embeds into a quotient of $B$, which is an obvious contradiction. If $\PP$ has height one, $B=A_\PP$ is a local principal ideal domain; if $p$ denotes a generator of its maximal ideal, then $J=IB$ is generated by $p^k$ for some $k$. The $p$-torsion in $B/J$ is then exactly $p^{k-1}A/p^kA$, which has dimension 1 qua $B/pB$-vector space, so cannot contain a copy of $(B/pB)^2$.  
\end{proof}

\begin{rem}
The conditions (being an UFD, or being of Krull dimension 2) cannot be dropped in Proposition \ref{deuxd}. Indeed, fix a field $K$ and consider the ring (of Krull dimension 2) $A=K[X,Y,T]/\langle T^3,T^2Y,TY^2\rangle$. Then for $\PP=\langle Y,T\rangle$, it admits an ideal isomorphic as $A$-module to $(A/\PP)^2$, generated by $(TY,T^2)$. In turn, it follows that the UFD $K[X,Y,T]$, of Krull dimension 3, also does not satisfy the conclusion of Proposition \ref{deuxd}. One also obtains similar examples with complete local rings localizing at 0 and completing, namely $K[\![X,Y,T]\!]/\langle T^3,T^2Y,TY^2\rangle$ and $K[\![X,Y,T]\!]$. Thus these complete local rings satisfy neither {\KK} nor {\KKK} of Theorem \ref{thmB}.
\end{rem}


\section{Proof of Theorem \ref{main}}\label{promain}

\begin{prop}\label{nona}
Let $A$ be a noetherian ring and $M$ an $A$-module. Suppose that $M$ is not minimax. Then $M$ has at least $2^{\aleph_0}$ submodules.
\end{prop}
\begin{proof}
This follows either from Theorem \ref{minidir}, which says that $M$ has a chain of submodules order-isomorphic to $(\mathbf{R},\le)$, or from Corollary \ref{exisola}, which says that the space of submodules of $M$ has no isolated point. Hence it has a continuous map onto a Cantor set (Lemma \ref{ontocantor}).
\end{proof}

\begin{proof}[Proof of Theorem \ref{main}]
First, if $M$ is countable then $2^{\aleph_0}$ is obviously an upper bound for the number of submodules.

If $M$ is not minimax then Proposition \ref{nona} implies that $M$ has at least $2^{\aleph_0}$ submodules. Let us now assume that $M$ is minimax, so that some quotient $Q$ of $M$ by a noetherian submodule, is artinian.

If $Q$ has Loewy dimension at least two, then for some maximal ideal $\MM$ of $A$, the Matlis dual $T(Q_\MM)$ is an $\widehat{A_\MM}$-module of Krull dimension at least two. In particular, it has at least $2^{\aleph_0}$ submodules, by Corollary \ref{crande}.

To show that Proposition {\LLL} is also necessary, it is enough to check that for every artinian module $M$ of infinite length, $M\times M$ has $2^{\aleph_0}$ submodules. We can suppose that $M$ is indecomposable, and therefore $M=M_\MM$ for some $\MM$. Thus the statement amounts by Matlis duality to show that if $R$ is a complete noetherian local ring with finite residual field and $M$ is a finitely generated $R$-module of infinite length, then $M\times M$ has $2^{\aleph_0}$ submodules. In turn, this is enough to check it when $M=A/\PP$ for some non-maximal prime ideal $\PP$, this is the contents of Proposition \ref{carre0}.

Let us now prove the positive part of the theorem: let $A$ be a noetherian ring and $M$ a minimax $A$-module satisfying Properties {\LL} and {\LLL}, and we have to show that $\Sub_A(M)$ is countable.

We first deal with the case when $M$ is artinian. There is a decomposition into a finite direct sum $M=\bigoplus M_\MM$ and every submodule of $M$ decomposes accordingly. Therefore it is enough to deal with the case where $M=M_\MM$. By Matlis duality, if $R=\widehat{A_\MM}$, we have to prove that if $V$ is a finitely generated $R$-module whose Matlis dual satisfies {\LL} and {\LLL}, then it has countably many submodules. This means that $V$ satisfies {\KK} and {\KKK}, and the conclusion that $V$ has countably many submodules follows from Proposition \ref{carre}.
 
Let us prove the general case. Let $N$ be a finitely generated submodule of $M$ such that $Q=M/N$ is artinian and satisfies {\LL} and {\LLL}. Since $N$ has countably many submodules, it is enough to check that for every submodule $N_0$ of $N$, the number of submodules $H$ of $M$ such that $H\cap N=N_0$ is countable. Let us prove the latter statement, assuming without lost of generality that $N_0=\{0\}$. Let $W$ be the union of finite length submodules of $M$. Note that $W\cap N$ has finite length, so in particular $W$ is also artinian and satisfies {\LL} and {\LLL}. Every submodule $H$ of $M$ such that $H\cap N=\{0\}$ is artinian and is therefore contained in $W$. By the previous case, $W$ has only countably many submodules, and we are done.
\end{proof}

Let us endow the set $\Sub(M)$ of submodules of an $A$-module $M$ with the topology induced by inclusion in $2^M$. The material of the previous section also proves (without assuming $A$ countable):

\begin{prop}\label{nosca}
Let $A$ be a noetherian ring. If $M$ fails to be minimax, or fails to satisfy Property {\LL} or {\LLL}, then $\Sub(M)$ is not scattered.
\end{prop}
\begin{proof}
If $M$ is not minimax, then $M$ has a subquotient isomorphic to an infinite countable direct sum of nonzero module, and the set of partial sums forms a Cantor set inside $\Sub(M)$.

If $M$ fails to satisfy {\LL}, by Proposition \ref{grandes} and Lemma \ref{phicont}, $\Sub(M)$ is not scattered. 

Suppose now that $M$ fails to satisfy {\LLL}. By assumption, $M$ admits a subquotient $L$ isomorphic to $N\times N$ for some artinian module $N$ of infinite length. We can suppose that $N$ has a single associated prime $\MM$. Since $\Sub(L)$ embeds as a closed subset of $\Sub(M)$, we can suppose $M=L$ and hence we can view $M$ as a $B=\widehat{A_\MM}$-module. Let $S$ be the Matlis dual of $N$: this is a finitely generated module of infinite length. Then $S$ has some subquotient isomorphic $B/\PP$ for some prime ideal $\PP$ of coheight 1. Hence passing to a subquotient again, we can suppose that $M$ is isomorphic to the Matlis dual of $(B/\PP)^2$. 

By Lemma \ref{1b}, $\Sub((B/\PP)^2)$, with the topology of \S\ref{toposub}, is non-scattered. By Lemma \ref{phicont}, the ``orthogonal" Matlis duality bijection $\Sub((B/\PP)^2)\to\Sub(L)$ is continuous. Hence $\Sub(L)$ is non-scattered. Since it embeds as a closed subset of $\Sub(M)$, we deduce that $\Sub(M)$ is non-scattered.
\end{proof}

I do not know whether the converse of Proposition \ref{nosca} holds (by Theorem \ref{main}, it holds when $A$ is countable).

\section{Meager modules}\label{megmo}

Again, $A$ denotes a (commutative associative unital) noetherian ring.

\subsection{Structure of meager modules}\label{s_meager}

We say that two $A$-modules $M_1,M_2$ are {\em disjoint} if $I_1+I_2=A$, where $I_i$ is the annihilator of $M_i$. A direct sum decomposition $M=\bigoplus M_i$ is called a {\em disjoint decomposition} if the $M_i$ are pairwise disjoint. Note that this implies a ``Chinese remainder Theorem", i.e.\ that every submodule $N$ of $M$ decomposes as $\bigoplus (N\cap M_i)$. 

We say that an $A$-module is {\em meager} if for every maximal ideal $\MM$ of $A$, the module $(A/\MM)^2$ is not a subquotient of $M$. 

\begin{lem}\label{meagdi}
Let $M$ be a meager module. Then the associated ideals of $M$ are pairwise disjoint.
\end{lem}
\begin{proof}
If $\PP$, $\QQ$ are distinct associated ideals, then $M$ has a submodule isomorphic to $(A/\PP)\times(A/\QQ)$. If they are not disjoint, then there exists a maximal ideal $\MM$ such that both $\PP$ and $\QQ$ are contained in $\MM$, and hence $M$ admits $(A/\MM)^2$ as subquotient, since the latter is a quotient of $(A/\PP)\times(A/\QQ)$.
\end{proof}

For an $A$-module $M$, define
\[M(\PP)=\{m\in M:\exists n\ge 1:\PP^nm=0\}.\]

\begin{lem}\label{assdis}
Let $M$ be a finitely generated $A$-module whose associated ideals are pairwise disjoint. 
Write
\[I=\prod_{\PP\in\Ass_A(M)}\PP\quad\text{and, for }\PP\in\Ass_A(M),\;I_\PP=\prod_{\mathcal{Q}\in\Ass_A(M)\smallsetminus\{\PP\}}\mathcal{Q}.\] Then there exists $n$ such that $I^nM=0$ and, for all $\PP\in\Ass_A(M)$ we have $M(\PP)=I_\PP^nM$. Moreover, $M=\bigoplus_{\PP\in\Ass_A(M)}M(\PP)$.
 \end{lem}
\begin{proof}
For all $\PP,\QQ$ distinct in $\Ass_A(M)$, we have $1\in\PP+\QQ$. Multiplying such relations, we obtain that for all $\PP,\QQ_1,\dots,\QQ_c$ with $\PP\neq\QQ_i$, and all $n$ we have $1\in\PP^n+(\prod\QQ_i)^n$; we freely use this below.

We now prove the result by induction on the cardinal $c$ of $\Ass_A(M)$; if $c=0$ then $M=0$ and the result holds.

We can replace $A$ by its quotient by the annihilator of $M$ to assume that $M$ is a faithful $A$-module, so the associated ideals of $M$ are, by the disjointness assumption, the minimal prime ideals of $A$.

So, for $\PP\in\Ass_A(M)$, the localization $A_\PP$ is an artinian local ring; hence there exists $m$ such that $\PP^mA_\PP=0$ (we choose $m$ working for all $\PP\in\Ass_A(M)$). Denoting by $K_\PP$ the kernel of $M\to M_\PP$, this implies that $\PP^mM\subset K_\PP$. Since $K_\PP\otimes A_\PP=0$, we have $\PP\notin\Ass_A(K_\PP)$.

In case $\Ass_A(M)=\{\PP\}$, we deduce $\Ass_A(K_\PP)=\emptyset$, so $K_\PP=0$ and hence $\PP^mM=0$, proving the result for $c=1$.

In general ($c\ge 1$), we deduce hence $\PP\notin\Ass_A(\PP^mM)$.

Applying this to other associated prime ideals of $M$, we obtain that for every $\QQ\in\Ass_A(M)\smallsetminus\{\PP\}$, $\QQ\notin I_\PP^mM$ and $\QQ I_\PP^mM=I_\PP^mM$. So $\Ass_A(I_\PP^mM)\subset\{\PP\}$.

Since $\PP^m+I_\PP^m=A$, we have $M=\PP^mM+I_\PP^mM$. Since they do not have any common associated ideals, the sum is direct. Decomposing $\PP^mM$ by induction (and using the case $c=1$), we obtain the result for $M$.
\end{proof}

\begin{prop}\label{assdis2}
Let $M$ be an $A$-module whose associated ideals are pairwise disjoint (e.g., a meager $A$-module, by Lemma \ref{meagdi}). Then \[M=\bigoplus_{\PP\in\Ass_A(M)}M(\PP),\] and $\Ass_A(M(\PP))=\{\PP\}$ for all $\PP\in\Ass_A(M)$. 

Conversely, given a subset $X$ of pairwise disjoint prime ideals of $A$ and for each $\PP\in X$ a meager $A$-module $M(\PP)$ with $\Ass_A(M(\PP))=\{\PP\}$, the direct sum $\bigoplus_{\PP\in X}M(\PP)$ is meager.
\end{prop}
\begin{proof}
The first statement follows from the finitely generated case (Lemma \ref{assdis}). The second one is straightforward.
\end{proof}

Since any disjoint direct sum of meager modules is meager, Proposition \ref{assdis2} reduces the study of meager modules to the case of modules with a single associated prime.

\begin{thm}\label{meager_stru}
Let $M$ be a meager $A$-module with a single associated prime $\PP$. Then exactly one of the following holds:
\begin{enumerate}[(a)]
\item\label{imea1} $M$ has nonzero finite length and its submodules form a chain (and $\PP$ is a maximal ideal);
\item\label{imea2} $\PP$ is a maximal ideal, and there exists a (unique) prime ideal $\QQ$ of coheight 1 in the complete local ring $\widehat{A_\PP}$ such that $B=A_\PP/\QQ$ is a discrete valuation ring and $M$ is isomorphic to $\mathrm{Frac}(B)/B$ as $A$-module;
\item\label{imea3} $\PP$ has coheight 1, the quotient ring $A/\PP$ is a Dedekind domain, and $M$ is a torsion-free module of rank 1 over $A/\PP$ (or equivalently, is isomorphic to some nonzero submodule of $\mathrm{Frac}(A/\PP)$).
\end{enumerate}
Conversely, any $A$-module in one of these cases is meager with only associated ideal $\PP$.

The module $M$ is artinian in Cases (\ref{imea1}) and (\ref{imea2}) but not in (\ref{imea3}).

 The module $M$ is finitely generated in Case (\ref{imea1}) but not in (\ref{imea2}) (in Case (\ref{imea3}) it can both be finitely generated or not, for instance with $M=A/\PP$ on the one hand and $M=\mathrm{Frac}(A/\PP)$ on the other hand).

In Case (\ref{imea3}), the annihilator of $M$ is equal to $\PP$.
In Cases (\ref{imea2}), the annihilator of $M$ as $\widehat{A_\PP}$-module is equal to $\QQ$.
\end{thm}

\begin{rem}
In Case (\ref{imea2}) of Theorem \ref{meager_stru}, the annihilator $\mathcal{W}$ of $M$ as $A$-module is a (non-maximal) prime ideal (the inverse image of $\mathcal{Q}$ in $A$). Beware that $\mathcal{W}$ does not necessarily have coheight $1$. Indeed, consider a countable field $K$, $A=K[x,y]$, $\PP=\langle x,y\rangle$ and $\widehat{A_\PP}=K[\![x,y]\!]$. Then $A$ has only many countably many ideals, and for each ideal of coheight 1 $\QQ'$ contained in $\PP$, the completion $(A/\QQ')_{\PP/\QQ'}$ has only finitely many minimal prime ideals, which leaves only finitely many possibilities for $\QQ$ and hence for the isomorphism type of $M$. 

On the other hand, $\widehat{A_\PP}=K[\![x,y]\!]$ has uncountably many distinct principal ideals (see Theorem \ref{grande}), and hence (being a UFD) has uncountably many prime ideals of height 1, which yields uncountably many possibilities, and hence, with countably many exceptions on $\mathcal{Q}$, the resulting $A$-modules $\mathrm{Frac}(B)/B$  are faithful.
\end{rem}

\begin{lem}\label{localmeager}
Let $M$ be a meager $A$-module and $S$ a multiplicative subset of $A$. Then $S^{-1}M$ is a meager $S^{-1}A$-module.
\end{lem}
\begin{proof}
If $S^{-1}M$ is not meager, we can suppose, replacing $M$ by a suitable finitely generated submodule if necessary, that $S^{-1}M$ has a quotient of the form $W^2$ for some simple $S^{-1}A$-module $W$. (If $W$ is simple as $A$-module, we are done but beware it is not automatic.) Let $P$ be the inverse image in $A$ of the annihilator in $S^{-1}A$ of any nonzero element of $W$. The image of $M$ in $W^2$ generates $W^2$ as $S^{-1}A$-module. In particular, it contains an $A$-submodule isomorphic to $(A/P)^2$. This implies that $M$ is not meager.
\end{proof}

\begin{lem}\label{dome}
Suppose that $A$ is a domain. Equivalences:
\begin{enumerate}[(i)]
\item\label{dome1} $A$ is a meager $A$-module;
\item\label{dome2} $\mathrm{Frac}(A)$ is a meager $A$-module;
\item\label{dome3} $A$ is a Dedekind domain.
\end{enumerate}
\end{lem}
\begin{proof}
Clearly (\ref{dome2}) implies (\ref{dome1}), and the converse is true because every finitely generated submodule of $\mathrm{Frac}(A)$ is contained in a cyclic submodule.

So let us prove the equivalence between (\ref{dome1}) and (\ref{dome3}).
Suppose that $A$ is a meager $A$-module and let us show that $A$ is Dedekind. First suppose that $A$ is a local domain with maximal ideal $\MM$. Then $A$ is Dedekind, that is, a discrete valuation ring, if and only $\MM/\MM^2$ is generated by a single element. This condition holds if $A$ is meager.

In general ($A$ maybe not local), $A$ Dedekind means that its localizations at all prime ideals are discrete valuation rings. Since all its localizations $A_\PP$ are meager as $A_\PP$-modules (by Lemma \ref{localmeager}), this follows from the local case. 

Conversely, assume that $A$ is Dedekind and let us show that $A$ is a meager $A$-module. First note that this is clear when $A$ is a principal ideal ring. Otherwise, $(A/\MM)^2$ is a subquotient of $A$ for some maximal ideal $\MM$. Localizing at $\MM$ and using flatness of $A_\MM$ as $A$-module, we obtain the same statement over the localization $A_\MM$, which is a discrete valuation ring, hence a principal ideal ring, and we reach a contradiction.
\end{proof}

\begin{lem}\label{mfgi}
Let $M$ be a meager finitely generated $A$-module of infinite length, with a single associated prime ideal $\PP$. Then $\PP$ has coheight 1 and $M$ is a torsion-free $(A/\PP)$-module of rank 1 (i.e., is isomorphic as $A$-module to a nonzero ideal of $A/\PP$).
\end{lem}
\begin{proof}
By Lemma \ref{lomega}, we have to prove that the ordinal length of $M$ is equal to $\omega$. Suppose by contradiction that the ordinal length of $M$ is $>\omega$; then $M$ has a quotient $M'$ with $\ell(M')=\omega+1$. As a quotient of $M$, all associated ideals of $M'$ contain $\PP$.
Since $A/\PP$ is a meager $A/\PP$-module, by Lemma \ref{dome}, $\PP$ has coheight 1. Since $M'$ is finitely generated and has infinite length, it has at least an associated prime ideal of positive coheight, and hence $\PP$ is the only possibility; since $M'$ is meager its associated prime ideals are pairwise disjoint and it follows that $\Ass_A(M')=\{\PP\}$. But this contradicts Lemma \ref{plusone}, which says that $M'$ has a simple submodule.
\end{proof}

\begin{proof}[Proof of Theorem \ref{meager_stru}]
All the additional statements are immediate. 

The ``conversely" statement is immediate in the first two cases; for the last one, we need to check that if $A$ is a Dedekind domain then $\mathrm{Frac}(A)$ is meager; this is done in Lemma \ref{dome}. 

Now let us prove the main statement, namely that every meager module $M$ with $\Ass_A(M)=\{\PP\}$ has the given form. 

Suppose that $\PP$ is not maximal and let us prove that we are in Case (\ref{imea3}). By Lemma \ref{dome}, $\PP$ has coheight 1 and $A/\PP$ is a Dedekind domain. To prove the result, it is enough to show that $\PP M=0$; this reduces to the finitely generated case and follows from Lemma \ref{mfgi}.

Suppose that $\PP$ is maximal. If $M$ has finite length, we are in Case (\ref{imea1}); assume that $M$ has infinite length and let us prove that we are in Case (\ref{imea2}). We claim that $M$ is artinian: indeed, if by contradiction we have a properly decreasing chain of submodules, in the quotient $M'$ by its intersection we have a decreasing chain of nonzero submodules with trivial intersection, but every nonzero submodule should contain the $\PP$-torsion (the set of elements killed by $\PP$, which in this case is reduced to a single simple submodule).

Since $\Ass_A(M)=\{\PP\}$, $M$ is naturally a module over the completion $\widehat{A_\PP}$, which is also meager. Since $M$ is artinian, the Matlis dual $T(M)$ is a finitely generated meager $\widehat{A_\PP}$-module, of infinite length. Since $M$ has a unique minimal nonzero submodule, $T(M)$ has a unique maximal proper submodule and hence is not decomposable as a nontrivial direct sum; thus, by Proposition \ref{assdis2}, $\Ass_{\widehat{A_\PP}}T(M)$ is a singleton $\{\mathcal{Q}\}$. Since $T(M)$ is finitely generated of infinite length, it has an associated prime ideal that is not maximal, so $\mathcal{Q}$ is not maximal. By Lemma \ref{dome}, $B=\widehat{A_\PP}/\QQ$ is a discrete valuation ring (and not a field). Applying the previous case to $T(M)$ over $\widehat{A_\PP}$, we are in Case (\ref{imea2}), which in this case implies that $T(M)$ is a free $B$-module of rank 1. So $\QQ M=0$, and applying Matlis duality over the discrete valuation ring $B$, for which an injective hull of the residual field is given by $\mathrm{Frac}(B)/B$, we deduce that $M$ is isomorphic as $B$-module, and hence as $A$-module, to $\mathrm{Frac}(B)/B$.
\end{proof}

\subsection{Counting submodules of meager modules}

A first consequence of Theorem \ref{meager_stru} is Corollary \ref{uniserial}.

\begin{proof}[Proof of Corollary \ref{uniserial}]
Clearly, the assumption implies that $M$ is meager with a single associated ideal. So we can apply Theorem \ref{meager_stru}, whose first case (finite length) is excluded by assumption. If we are in Case (\ref{imea2}) of Theorem \ref{meager_stru}, then we obtain (\ref{icha2}). If we are in Case (\ref{imea3}), we first observe that $A/\PP$ has to be local, since when $A/\PP$ is not local then its ideals do not form a chain, as we see by taking two distinct maximal ideals. So $A/\PP$ is a discrete valuation ring, and it is immediate that every nonzero proper submodule of the fraction field is isomorphic to the ring as a module.
\end{proof}

\begin{thm}\label{dedenb}
Let $A$ be a Dedekind domain (which is not a field); let $\alpha$ be the cardinal of its set of maximal ideals. Let $M$ be a nonzero submodule of $\mathrm{Frac}(A)$. Writing $\mathrm{Frac}(A)/A=\bigoplus_{\PP}\mathrm{Frac}(A)/A_\PP$, let $\mathcal{S}$ be the set of $\mathcal{P}$ such that the projection of $M$ on $\mathrm{Frac}(A)/A_\PP$ is nonzero, and let $\gamma$ be the cardinal of $\mathcal{S}$.

Then the number of submodules of $M$ is $\max(\alpha,2^{\gamma},\aleph_0)$. In particular, if $M$ is minimax, then this is $\max(\alpha,\aleph_0)$.
\end{thm}
\begin{proof}
We start with the case $M=A$; we have to show that the cardinal $\delta$ of the number of ideals of $A$ is $\beta=\max(\alpha,\aleph_0)$. Clearly this cardinal is infinite and at least equal to the cardinal of the set of maximal ideals of $A$, so $\delta\ge\beta$. Let us show the reverse inequality $\delta\le\beta$. For every maximal ideal $\MM$, every $\MM$-primary ideal contains $\MM^n$ for some $n$, and the set of ideals of $A/\MM^n$ is a finite chain for all $n$, hence is finite. So the set of $\MM$-primary ideals has cardinal $\aleph_0$. Since every ideal is a finite intersection of primary ideals, it is determined by a finite set $F$ of maximal ideals an a choice of $\MM$-primary ideal for every $\MM\in F$, so the number of ideals is $\le\beta$.

Now let $M$ be arbitrary. Multiplying $M$ by a nonzero element of $\mathrm{Frac}(A)$ does not change its number of submodules, and does only affect $\mathcal{S}$ by a finite set, and in particular does not affect $\max(2^{\gamma},\aleph_0)$. So we can suppose that $M$ contains $A$.

Clearly $\max(\alpha,\aleph_0)$ is a lower bound for the number of submodules of $M$. Since $M/A$ is a direct sum of $\gamma$ nonzero submodules, it contains at least $2^{\gamma}$ submodules. This complete the lower bound.

If $\gamma$ is finite, the number of submodules of $M/A$ is (at most) countable, and the same argument shows that the number of submodules containing a given nonzero ideal of $A$ is countable, and since the number of ideals is $\max(\aleph_0,\alpha)$, we deduce the upper bound $\alpha\max(\aleph_0,\alpha)=\max(\aleph_0,\alpha)$.

If $\gamma$ is infinite, the number of submodules of $\mathrm{Frac}(A)/A_\PP$ is $\aleph_0$, and hence the number of submodules of $M/A$ is $\le\aleph_0^{\gamma}$. Since $\gamma$ is infinite $2^\gamma=2^{\aleph_0\gamma}=(2^{\aleph_0})^\gamma$, we have $\aleph_0^{\gamma}=2^\gamma$. For the same reason, the number of submodules of $M$ containing any given nonzero ideal of $A$ is $\le 2^\gamma$. Hence the number of submodules is $\le \max(\alpha,\aleph_0)2^\gamma=\max(\alpha,\aleph_0,2^\gamma)$.

For the last statement, just observe that if $M$ is minimax then $\gamma$ is finite.
\end{proof}

\begin{proof}[Proof of Theorem \ref{prop3}]
Since $M$ is minimax, it has finitely many associated ideals. Write the disjoint (finite) direct sum $M=\prod_\PP M(\PP)$ as in Proposition \ref{assdis2}. Since any submodule decomposes accordingly, we have, denoting by $\Sub(M)$ the set of its submodules, $\Sub(M)=\prod_\PP\#\Sub(M(\PP))$. We use Theorem \ref{meager_stru} to get the following discussion.

If $M$ has finite length then $\Sub(M(\PP))$ is a finite chain and hence $\Sub(M)$ is finite. 

If $M(\PP)$ is artinian of infinite length, the cardinal of $\Sub(M(\PP))$ is $\aleph_0$. Hence is $M$ is artinian (or equivalently all its associated ideals are maximal) and has infinite length, then the cardinal of $\Sub(M)$ is $\aleph_0$.

If $M(\PP)$ is non-artinian, the cardinal $\Sub(M(\PP))$ is, by Theorem \ref{dedenb}, equal to $\max(\beta_\PP,\aleph_0)$, where $\beta_\PP$ is the cardinal of the set of maximal ideals of $A/\PP$. We deduce, in this case that, denoting $\beta=\max_\PP\beta_\PP$, that the cardinal of $\Sub(M)$ is $\max(\beta,\aleph_0)$.

In the non-meager case, $M$ has a subquotient of the form $K^2$ for some quotient field $K$ of $A$, and by assumption $K$ has cardinal $\alpha$, so $K^2$ has exactly $\alpha$ submodules, and hence $M$ has at least $\alpha$ submodules. 

Since $A$ is noetherian and $M$ is minimax, it is easily checked that $M$ has a countable generating family as well as its submodules, and hence the cardinal of the set of submodules of $M$ is $\le\alpha^{\aleph_0}$.

By Lemma \ref{crande} and Matlis duality, if {\LL} fails then $M$ has $\ge \alpha^{\aleph_0}$ submodules. Similarly, by Proposition \ref{carre0} and Matlis duality, if {\LLL} fails then $M$ has $\ge \alpha^{\aleph_0}$ submodules.

Now assume that Properties {\LL} and {\LLL} hold and let us show that $M$ has $\le\alpha$ submodules. Let $N$ be a finitely generated submodule of $M$ such that $M/N$ is artinian. Since a finitely generated module has $\le\alpha$ submodules, it is enough to show that the number of submodules with given intersection $J$ with $N$ is $\le\alpha$. Working in $M/J$ reduces to the case $J=0$. In other words, we have to show that submodules with zero intersection with $N$ is $\le\alpha$. Since such submodules are artinian, they are contained in the union of all finitely generated submodules of $M$, which is artinian. So we are reduced to the case when $M$ is artinian. Then $M$ is a finite product of artinian modules with a single associated ideal, which reduces to the case when $M$ has a single associated ideal $\PP$, which is maximal, and hence is naturally an $\widehat{A_\PP}$-module. Then we are reduced to the statement that if $B$ is a complete local ring with residual field of cardinal $\alpha$ and $M$ is an artinian $B$-module satisfying {\LL} and {\LLL} then $B$ has $\le\alpha$ submodules. This follows from Proposition \ref{carre} by Matlis duality.
\end{proof}



\end{document}